\documentclass[12pt]{amsart}
\usepackage{amsmath}
\usepackage{amsfonts}
\usepackage{amssymb}
\usepackage[all,cmtip]{xy}           
\usepackage{bm}
\usepackage{bbm}
\usepackage{bbding}
\usepackage{txfonts}
\usepackage{amscd}
\usepackage{xspace}
\usepackage[shortlabels]{enumitem}
\usepackage{ifpdf}

\ifpdf
  \usepackage[colorlinks,final,backref=page,hyperindex]{hyperref}
\else
  \usepackage[colorlinks,final,backref=page,hyperindex,hypertex]{hyperref}
\fi
\usepackage{tikz}
\usepackage[active]{srcltx}

\usepackage{tikz-cd}
\usepackage{tikz}

\makeatletter

\newtheorem{thm}{Theorem}[section]
\newtheorem{lem}[thm]{Lemma}
\newtheorem{cor}[thm]{Corollary}
\newtheorem{pro}[thm]{Proposition}
\theoremstyle{definition}
\newtheorem{ex}[thm]{Example}
\newtheorem{rmk}[thm]{Remark}
\newtheorem{defi}[thm]{Definition}

\newcommand{\nc}{\newcommand}
\newcommand{\delete}[1]{}

\nc{\mlabel}[1]{\label{#1}}  
\nc{\mcite}[1]{\cite{#1}}  
\nc{\mref}[1]{\ref{#1}}  
\nc{\meqref}[1]{\eqref{#1}}  
\nc{\mbibitem}[1]{\bibitem{#1}} 

\delete{
\nc{\mlabel}[1]{\label{#1}{\hfill \hspace{1cm}{\bf{{\ }\hfill(#1)}}}}
\nc{\mcite}[1]{\cite{#1}{{\bf{{\ }(#1)}}}}  
\nc{\mref}[1]{\ref{#1}{{\bf{{\ }(#1)}}}}  
\nc{\meqref}[1]{\eqref{#1}{{\bf{{\ }(#1)}}}}  
\nc{\mbibitem}[1]{\bibitem[\bf #1]{#1}} 
}

\DeclareMathOperator{\im}{Im}

\newcommand {\emptycomment}[1]{}

\delete{

}

\nc{\oprn}{\theta}

\nc{\calo}{\mathcal{O}}
\nc{\oop}{$\mathcal{O}$-operator\xspace}
\nc{\oops}{$\mathcal{O}$-operators\xspace}
\nc{\mrho}{{\bm{\varrho}}}
\nc{\bfk}{\mathbf{K}}
\nc{\invlim}{\displaystyle{\lim_{\longleftarrow}}\,}
\nc{\ot}{\otimes}

\nc{\desc}{descendent\xspace}
\nc{\eval}[1]{\Big|_{#1}}

\newcommand{\B }{\mathfrak{B}}
\newcommand{\lon }{\,\rightarrow\,}
\newcommand{\be }{\begin{equation}}
\newcommand{\ee }{\end{equation}}

\newcommand{\g}{\mathfrak g}

\newcommand{\G}{\mathbb G}

\nc{\RR}{\mathbb{R}}

\nc{\CC}{\mathbb{C}}

\newcommand{\huaB}{\mathcal{B}}

\newcommand{\huaA}{\mathcal{A}}

\newcommand{\huaG}{\mathcal{G}}

\newcommand{\huaH}{\mathcal{H}}

\newcommand{\CWM}{C^{\infty}(M)}

\newcommand{\frkB}{\mathfrak B}
\newcommand{\frkC}{\mathfrak C}
\newcommand{\frkD}{\mathfrak D}

\newcommand{\frkX}{\mathfrak X}

\newcommand{\Courant}[1]{\left\llbracket  #1\right\rrbracket }

\newcommand{\br}[1]{   [ \cdot,    \cdot  ]   }
\newcommand{\id}{\mathsf{id}}

\newcommand{\Der}{\mathrm{Der}}

\newcommand{\Ad}{\mathrm{Ad}}
\newcommand{\Aut}{\mathrm{Aut}}

\newcommand{\gl}{\mathfrak {gl}}

\nc{\CV}{\mathbf{C}}

\begin{document}

\title[Integration and geometrization of Rota-Baxter Lie algebras]{Integration and geometrization of Rota-Baxter Lie algebras}

\author{Li Guo}
\address{
Department of Mathematics and Computer Science,
         Rutgers University,
         Newark, NJ 07102}
\email{liguo@rutgers.edu}

\author{Honglei Lang}
\address{Department of Applied Mathematics, China Agricultural University, Beijing, 100083, China
}
\email{hllang@cau.edu.cn}

\author{Yunhe Sheng}
\address{Department of Mathematics, Jilin University, Changchun 130012, Jilin, China}
\email{shengyh@jlu.edu.cn}

\date{\today}

\begin{abstract}
This paper first introduces the notion of a Rota-Baxter operator (of weight $1$) on a Lie group so that its differentiation gives a Rota-Baxter operator on the corresponding Lie algebra. Direct products of Lie groups, including the decompositions of Iwasawa and Langlands, carry natural Rota-Baxter operators. Formal inverse of the Rota-Baxter operator on a Lie group is precisely the crossed homomorphism on the Lie group, whose tangent map is the differential operator of weight $1$ on a Lie algebra. A factorization theorem of Rota-Baxter Lie groups is proved, deriving directly on the Lie group level, the well-known global factorization theorems of Semenov-Tian-Shansky in his study of integrable systems. As  geometrization, the notions of Rota-Baxter Lie algebroids and Rota-Baxter Lie groupoids are introduced, with the former a differentiation of the latter. Further, a Rota-Baxter Lie algebroid naturally gives rise to a post-Lie algebroid, generalizing the well-known fact for Rota-Baxter Lie algebras and post-Lie algebras. It is shown that the geometrization of a Rota-Baxter Lie algebra or a Rota-Baxter Lie group can be realized by its action on a manifold. Examples and applications are provided for these new notions.
\end{abstract}

\subjclass[2010]{22A22, 22E60, 17B38, 17B40, 81R12, 58C35}

\keywords{Rota-Baxter Lie group,  Rota-Baxter Lie groupoid, Rota-Baxter Lie algebroid, post-Lie algebroid, integration}

\maketitle

\tableofcontents

\section{Introduction}

We define Rota-Baxter operators on Lie groups, Lie algebroids and Lie groupoids to give the integration and geometrization of Rota-Baxter operators on Lie algebras. The relationship of Rota-Baxter Lie algebroids with post-Lie algebroids is established.

\subsection{Lie groupoids and Lie algebroids}

Lie groups and Lie algebras are fundamental notions in mathematics. Their relationship, with a Lie algebra as the tangent of a Lie group, plays a major role in the study of geometry, algebra and mathematical physics. See for example~\cite{H,K}.

Lie groupoids incorporate diverse objects, including Lie groups, manifolds, Lie group actions and equivalence relations. They have played important roles in mathematical physics and noncommutative geometry. In particular, Lie groupoids unify internal and external symmetries and are used to describe singular quotient spaces in noncommutative geometry.
The notion of a Lie algebroid was introduced by Pradines~\cite{Pra} in 1967, as a generalization of Lie algebras and tangent bundles. Just like Lie algebras being the
infinitesimal objects of Lie groups, Lie algebroids are the infinitesimal objects of Lie groupoids. Different from the case of Lie algebras and Lie groups, in general there is an obstruction class in integrating a Lie algebroid to a Lie groupoid \cite{crainic}. A natural class of Lie algebroids and Lie groupoids are given by the actions of Lie algebras and Lie groups on manifolds. See \mcite{Mkz:GTGA} for the general theory of Lie algebroids and Lie groupoids.

These notions and operations can be summarized in the diagram
\begin{equation}
\begin{split}
\xymatrix{
	\text{Lie algebras} \ar^{\text{geometrization}}[rrr] &&& \text{Lie algebroids}&&&\text{Lie algebras}\ar_{\text{action}}[lll]\\
	\text{Lie groups} \ar^{\text{geometrization}}[rrr] \ar^{\text{differentiation}}[u] &&& \text{Lie groupoids} \ar^{\text{differentiation}}[u]&&&\text{Lie groups.}\ar_{\text{action}}[lll] \ar^{\text{differentiation}}[u]
}
\end{split}
\mlabel{eq:reln}
\end{equation}

\subsection{Classical Yang-Baxter equations and Rota-Baxter Lie algebras}

The classical Yang-Baxter equation (CYBE) arose from the study of
inverse scattering theory in the 1980s and then was recognized as the
``semi-classical limit" of the quantum Yang-Baxter equation following the works of C.~N. Yang~\mcite{Ya} and R.~J. Baxter~\mcite{BaR}.
CYBE is further related to classical integrable
systems and quantum groups~\cite{CP}.

An important method  in studying the CYBE is the interpretation of it in various operator forms, beginning with the pioneering work of  Semenov-Tian-Shansky~\mcite{STS} who showed
that if there exists a nondegenerate symmetric invariant bilinear
form on a Lie algebra $(\g,[\cdot,\cdot]_\g)$ and if a solution of
the CYBE is skew-symmetric, then the solution can be equivalently expressed
as a linear operator $B:\g\to \g$ satisfying the operator identity
\begin{equation}\notag
	[B(u), B(v)]_\g=B([B(u),v]_\g)+B([u,B(v)]_\g),\;\;\forall u, v\in
	\g.\end{equation}
Also introduced in~\cite{STS} is the {\bf modified Yang-Baxter equation}
\begin{equation}
	\label{eq:mybe}
	 [R(u),R(v)]_\g=R([R(u),v]_\g)+R([u,R(v)]_\g)-[u,v]_\g, \quad \forall u, v\in \g,
\end{equation}
from which the author obtained an Infinitesimal Factorization Theorem for the Lie algebra $\g$ and, after integration, his well-known Global Factorization Theorem for the corresponding Lie group, with important applications to integrable systems~\mcite{FRS,Iz,Li,LT,RS,RS1,RS2}.
Under the transformation $R=\id+2B$, the operator $R$ satisfies the modified Yang-Baxter equation if and only if the operator $B$ is the special case when $\lambda=1$ of a {\bf Rota-Baxter operator of weight $\lambda$}, characterized by the operator identity
\begin{equation}
[B(u), B(v)]_\g=B([B(u),v]_\g)+B([u,B(v)]_\g)+\lambda B([u,v]_\g),\;\;\forall u, v\in
\g,
\mlabel{eq:rbo}
\end{equation}
for any scalar $\lambda$.
A Lie algebra $(\g,[\cdot,\cdot]_\g)$ with a Rota-Baxter operator $B$ of weight $\lambda$ is
called a {\bf Rota-Baxter Lie algebra} of weight $\lambda$, denoted by $(\g,[\cdot,\cdot]_\g,B)$.

As a remarkable coincidence, the associative analogs of Rota-Baxter operators had been introduced by G. Baxter in 1960 in his probability study~\mcite{Bax} and pursued further by Atkinson, Cartier and foremost Rota from the perspectives of analysis and combinatorics~\cite{At,Ca,Ro1,Ro2}. The remarkable renascence of the subject in this century has led to broad applications, especially in the Connes-Kreimer approach to renormalization of quantum field theory~\cite{CK,Gub}.
	
A Rota-Baxter Lie algebra with weight, and more generally an $\mathcal{O}$-operator~\cite{Bor,Ku}, naturally gives rise to a pre-Lie algebra or a post-Lie algebra which has its origin in a study of operads~\mcite{Val} as a special case of the splitting of Lie algebras~\mcite{BBGN}. Pre-Lie algebras and post-Lie algebras  play important roles  in integrable systems and numerical integrations~\mcite{BGN,Munthe-Kaas-Lundervold}.
See \mcite{Bu,BG,Burde16,Burde19} for more details about pre-Lie algebras and  post-Lie algebras. Recently, the notion of post-Lie algebroids was introduced and applied to numerical integrations~\cite{Munthe-Kaas-Lundervold,MSV}.

\subsection{Rota-Baxter operators for other Lie structures}
In light of the critical roles played by Lie groups, Lie groupoids and Lie algebroids, as well as their relationship with Lie algebras as shown in the diagram~(\mref{eq:reln}) on the one hand, and Rota-Baxter Lie algebras and post-Lie algebras on the other, it is desirable to combine these notions and to extend the relationship among the Lie structures such as the one in~\meqref{eq:reln}. In particular, the fundamental Global Factorization Theorem of Semenov-Tian-Shansky for a Lie group was obtained from integrating his Infinitesimal Factorization Theorem for a Lie algebra, making use of the modified Yang-Baxter equation (equivalently, a Rota-Baxter operator of weight $1$). Thus it is natural to ask whether is a Rota-Baxter operator on the Lie group, so that the Global Factorization Theorem can be proved directly on the Lie group level. These are the questions we address in this paper.

The obvious challenge in giving the notion of a Rota-Baxter Lie group is that a Lie group does not have the linear structure required for the original Rota-Baxter relation in~(\mref{eq:rbo}). We use the adjoint map
to define an operator on a Lie group with a suitable operator identity~\meqref{RBgroup} that has properties similar to Rota-Baxter operators of weight $1$ on Lie algebras, and such that the differentiation of such an operator on a Lie group is a Rota-Baxter operator of weight $1$ on the corresponding Lie algebra, thereby justifying naming it a Rota-Baxter operator (of weight $1$) on a Lie group. The notion of a Rota-Baxter operator of weight $-1$ on a Lie group is also defined though our focus in this paper will be on the weight $1$ case. We also show that the well-known notion of a crossed homomorphism on a group serves the purpose of a differential operator of weight $1$ on a Lie group since it is a formal inverse of the Rota-Baxter operator and its tangent map is the usual differential operator of weight $1$ on a Lie algebra.
We further show that the Rota-Baxter operator on a Lie group shares some important properties of the Rota-Baxter operator on Lie algebra and give the aforementioned Global Factorization Theorem of Semenov-Tian-Shansky for Lie groups.

We further generalize the notion of a Rota-Baxter operator to Lie algebroids by equipping a Lie algebroid with a bundle map satisfying a Rota-Baxter type identity. The well-known fact that a Rota-Baxter operator on a Lie algebra gives rise to a post-Lie algebra finds its analog for Rota-Baxter operators on Lie algebroids. We show that the action Lie algebroid of the induced Lie algebra by a Rota-Baxter operator always admits an action post-Lie algebroid structure, which can be viewed as a particular case of    \cite[Theorem 4.4]{MSV}.
By geometrizing the notion of a Rota-Baxter operator on a Lie group, we obtain the notion of a Rota-Baxter operator on a Lie groupoid and show that the differentiation of a Rota-Baxter operator on a Lie groupoid is a Rota-Baxter operator on a Lie algebroid, as expected.

We summarize the constructions and relations in diagram~\eqref{eq:summary}, enriching and extending diagram~\eqref{eq:reln}. The italic terms and dotted arrows are the ones introduced in this paper.

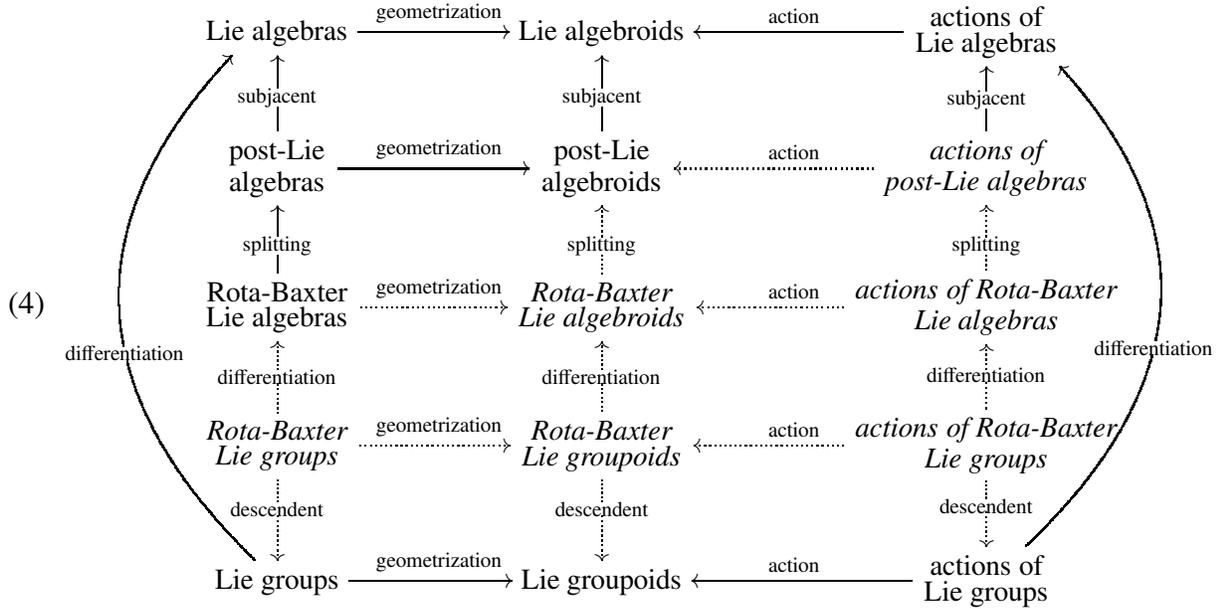
\begin{equation}
\begin{split}
\xymatrix{
{\text{\small{Lie algebras}}}\ar^{\text{geometrization}}[rr] && {\text{\small{Lie algebroids}}} && {\text{\small actions of} \atop\text{\small{Lie algebras}}}\ar_{\text{action}}[ll] \\
\text{\small post-Lie}\atop\text{\small algebras} \ar^{\text{geometrization}}[rr] \ar[u]!U|(.45){\text{subjacent}}  && {\text{\small  post-Lie}\atop \text{\small algebroids}} \ar[u]!U|(.45){\text{subjacent}} && \text{\em\small  actions of} \atop \text{\em\small  post-Lie algebras} \ar@{.>}_{\text{action}}[ll] \ar[u]!U|(.4){\text{subjacent}}  \\
{\text{\small Rota-Baxter} \atop \text{\small Lie algebras}} \ar@{.>}^{\text{geometrization}}[rr] \ar[u]!U|(.35){\text{splitting}}   && {\text{\em\small  Rota-Baxter} \atop \text{\it\small  Lie algebroids}} \ar@{.>}[u]!U|(.35){\text{splitting}}
&&  \text{\it\small  actions of Rota-Baxter} \atop \text{\it\small  Lie algebras} \ar@{.>}_{\text{action}}[ll] \ar@{.>}[u]!U|(.35){\text{splitting}} \\
\text{\it\small  Rota-Baxter}\atop \text{\it\small  Lie groups}\ar@{.>}[d]!U|(.55){\text{\desc}} \ar@{.>}[u]!U|(.4){\text{differentiation}} \ar@{.>}^{\text{geometrization}}[rr] && \text{\it\small  Rota-Baxter}\atop \text{\it\small  Lie groupoids}\ar@{.>}[d]!U|(.55){\text{\desc}}\ar@{.>}[u]!U|(.4){\text{differentiation}} && \text{\it\small  actions of Rota-Baxter} \atop \text{\it\small  Lie groups} \ar@{.>}[d]!U|(.6){\text{\desc}}\ar@{.>}_{\text{action}}[ll] \ar@{.>}[u]!U|(.4){\text{differentiation}}\\
{\text{\small{Lie groups}}}\ar^{\text{geometrization}}[rr] \ar@/^{5pc}/[uuuu]!U|(.4){\text{differentiation}}&& {\text{\small{Lie groupoids}}} && {\text{\small actions of} \atop\text{\small Lie groups}}\ar_{\text{action}}[ll] \ar@/_5.5pc/[uuuu]!U|(.4){\text{differentiation}}\\
}
\end{split}
\label{eq:summary}
\end{equation}

\subsection{Outline of the paper}
The paper is organized as follows. In Section~\mref{sec:rblg}, we first introduce the notion of a Rota-Baxter operator on a Lie group and show that projections to a direct factor Lie group gives a Rota-Baxter operator, leading to examples from the Iwasawa decomposition and Langlands decomposition.
We then justify this notion by showing that the differentiation of a Rota-Baxter operator on a Lie group is a Rota-Baxter operator on a Lie algebra (Theorem~\mref{main1}). A Rota-Baxter operator on a Lie group is also characterized by its graph, and taking the differentiation is also compatible with the derived product from a Rota-Baxter action. Differential Lie groups are also introduced.

In Section~\mref{sec:fact}, we establish a factorization theorem for Rota-Baxter Lie groups (Theorem~\mref{thm:factgp}) which derives the factorization theorem of Semenov-Tian-Shansky for Lie groups.

In Section~\mref{sec:oid}, we introduce the geometric aspect into our study. First a Rota-Baxter Lie algebroid is defined for a vector bundle, recovering a Rota-Baxter Lie algebra when the vector bundle reduces to a vector space. A Rota-Baxter Lie algebroid is shown to give a post-Lie algebroid introduced in~\mcite{Munthe-Kaas-Lundervold} (Theorem~\mref{thm:rbpost}). In particular, the notions of actions of Rota-Baxter Lie algebras and of post-Lie algebras on manifolds are introduced to produce Rota-Baxter Lie algebroids and post-Lie algebroids. Next we introduce the notion of a Rota-Baxter Lie groupoid and show that it is  the integration of a Rota-Baxter Lie algebroid  (Theorem~\mref{thm:algebroid}). The notion of an action  of a Rota-Baxter Lie group  on a manifold  is introduced, which gives rise to an action Rota-Baxter Lie groupoid and is compatible with the action Rota-Baxter Lie algebroid under differentiation.

\section{Rota-Baxter Lie groups and differential Lie groups}
\mlabel{sec:rblg}

In this section, we introduce the notions of   Rota-Baxter Lie groups and  differential Lie groups and show that they are the integrations of Rota-Baxter Lie algebras (of weight $1$) and differential Lie algebras (of weight $1$) respectively.

\subsection{Definition of Rota-Baxter Lie groups and examples}
We first give the notion of a Rota-Baxter Lie group and provide some examples.

Let $G$ be a group. For any $g\in G,$ define the {\bf adjoint action}
$$\Ad_g:G\longrightarrow G, \quad \Ad_g h:= ghg^{-1},\quad \forall~ h\in G.$$
The following two formulas will be frequently applied in the sequel:
\begin{eqnarray}
\mlabel{eq:f1}\Ad_g(h_1h_2)&=&\Ad_g h_1 \Ad_g h_2,\quad \forall~ h_1, h_2\in G,\\
\mlabel{eq:f2}\Ad_{g_1}\circ\Ad_{g_2}&=&\Ad_{g_1g_2},\quad \forall~ g_1, g_2\in G.
\end{eqnarray}

Here is a general notion.
\begin{defi}
A {\bf Rota-Baxter group} is a group $G$ with a map $\B:G\longrightarrow G$ such that
\begin{eqnarray}\mlabel{RBgroup}
\B(g_1)\B(g_2)=\mathfrak{B}(g_1\Ad_{\B(g_1)} g_2),\qquad \forall~ g_1,g_2\in G,
\end{eqnarray}
called the {\bf Rota-Baxter relation} for  groups.
\end{defi}

When $G$ is an abelian group, Eq.~\eqref{RBgroup} states that $\B$ is a group homomorphism.
We will mainly consider Lie groups in this paper.

\begin{defi}
Let $G$ be a Lie group. A smooth map $\B:G\longrightarrow G$ satisfying relation \eqref{RBgroup} is called a {\bf Rota-Baxter operator} on $G$.   A {\bf Rota-Baxter Lie group} is a Lie group  equipped with a Rota-Baxter operator.
\end{defi}

\begin{defi}
Let $(G,\B)$ and $(G',\B')$ be Rota-Baxter Lie groups. A smooth map $\Phi:G\longrightarrow G'$ is a {\bf Rota-Baxter Lie group homomorphism} if $\Phi$ is a Lie group homomorphism such that $$\Phi\circ \B=\B'\circ \Phi.$$
\end{defi}

The scalar multiplication by $-1$ on a Lie algebra is a Rota-Baxter operator of weight $1$. Correspondingly for a Lie group, we have
\begin{ex}
The inverse map $(\cdot)^{-1}:G\longrightarrow G$ is a Rota-Baxter operator on a Lie group $G$. Indeed the right hand side of Eq.~\eqref{RBgroup} equals to \[(\Ad_{g_1^{-1}} g_2)^{-1} g_1^{-1}=g_1^{-1} g_2^{-1} g_1 g_1^{-1}=g_1^{-1} g_2^{-1},\] which is exactly the left hand side of Eq.~\eqref{RBgroup}.
\end{ex}

For a Lie algebra $\g$, if $B:\g\to \g$ is a Rota-Baxter operator of weight $1$, then $-\id-B:\g\to \g$ is also a Rota-Baxter operator of weight $1$ on $\g$.  For a Lie group we have a similar result, noting that $-\id-B=-\id+B(-\id)$.
\begin{pro}
\mlabel{pro:adj}
If $\B:G\to  G$ is a Rota-Baxter operator on a Lie group $G$, then $\tilde{\B}: G\to G$ defined by $\tilde{\B}(g)=g^{-1}\B(g^{-1})$ is also a Rota-Baxter operator on $G$.
\end{pro}
\begin{proof}
We show that $\tilde{\B}$ satisfies \eqref{RBgroup}. In fact, for any $g,h\in G$,
\begin{eqnarray*}
\tilde{\B}(g\Ad_{\tilde{\B}(g)} h)&=&\tilde{\B}(gg^{-1}\B(g^{-1})h\B(g^{-1})^{-1}g)\\ &=&g^{-1}\B(g^{-1})h^{-1}\B(g^{-1})^{-1}\B(g^{-1}\B(g^{-1})h^{-1}\B(g^{-1})^{-1})\\ &=&g^{-1}\B(g^{-1})h^{-1}\B(g^{-1})^{-1}\B(g^{-1})\B(h^{-1})\\ &=&\tilde{\B}(g)\tilde{\B}(h),
\end{eqnarray*}
where for the second to the last equation we applied the Rota-Baxter relation~\eqref{RBgroup} of $\B$.
\end{proof}

The construction of solutions for the modified Yang-Baxter equation in \mcite{STS} motivated the following
important class of Rota-Baxter Lie algebras.

\begin{lem}\mlabel{lem:RBP}
Let $\g$ be an arbitrary Lie algebra and let $\g_{+}$ and $\g_{-}$ be its Lie subalgebras such that $\g=\g_+\oplus \g_{-}$ as vector spaces.
Denote by $P_+$ and $P_-$ the projections from $\g$ to $\g_+$ and $\g_-$ respectively. Then      $-P_+$ and $-P_-$ are Rota-Baxter operators of weight $1$.
\end{lem}

\begin{proof}
By the same argument as in the case of associative algebras~\mcite{Gub}, the projections $P_+$ and $P_-$ are Rota-Baxter operators of weight $-1$. Thus $-P_+$ and $-P_-$ have weight $1$.
\end{proof}

As an analog for Lie groups, we have

\begin{lem}\label{lem:RBPG}
Let $G$ be a Lie group and $G_{+},G_{-}$ be two Lie subgroups such that $G=G_{+} G_{-}$ and $G_+\cap G_-=\{e\}$. Define $\B:G\to G$ by
$$\B(g)=g_{-}^{-1},\quad \forall g=g_{+}g_{-},\quad \mbox{where}\quad g_{+}\in G_{+},g_-\in G_-.$$ Then $(G,\B)$ is a Rota-Baxter Lie group.
\mlabel{lem:rbpg}
\end{lem}

\begin{proof}
Let $g=g_+g_-$ and $h=h_+h_-$ be two elements in $G$ with $g_{+},h_{+}\in G_{+}$ and $g_-,h_-\in G_-$. Then by the fact that $G_{+}$ and $G_-$ are Lie subgroups, we have
\begin{eqnarray*}
\B(g\Ad_{\B(g)} h)=\B(g_+g_- g_{-}^{-1} h_+h_-g_{-})=\B(g_+h_+h_-g_-)= (h_-g_-)^{-1}=\B(g)\B(h).
\end{eqnarray*}
 Hence $\B$ satisfies \eqref{RBgroup}.
\end{proof}

\begin{rmk}
It is also natural to consider the inverse of the projection to the first factor. However, unlike the case of Lie algebras, the inverse of the projection to the first factor is not a Rota-Baxter operator of weight $1$. As we will see in Remark~\ref{rmk:relation}, the projection to the first factor is a Rota-Baxter operator of weight $-1$.
\end{rmk}

Lemma \ref{lem:RBPG}  provides a large number of Rota-Baxter Lie groups.

\begin{ex}
\begin{enumerate}
\item
By the Gram-Schmidt decomposition of matrices, we have the global decomposition of  $\mathrm{SL}(n,\mathbbm{C})$, the space of complex matrices with determinant $1$:
\[\mathrm{SL}(n,\mathbbm{C})=\mathrm{SU}(n)\mathrm{SB}(n,\mathbbm{C}),\]
 where $\mathrm{SU}(n)$ is the space of unitary matrices with determinant $1$ and $\mathrm{SB}(n,\mathbbm{C})$ consists of all upper triangular matrices in $\mathrm{SL}(n,\mathbbm{C})$ with positive entries on the diagonal.

 Then $(\mathrm{SL}(n,\mathbbm{C}), \B)$ is a Rota-Baxter Lie group, where $\B(ub)=b^{-1}$, for $u\in \mathrm{SU}(n)$ and $b\in \mathrm{SB}(n,\mathbbm{C})$.

More generally, for the Iwasawa decomposition~\cite[VI.4]{K} $G = KAN$ of a semisimple group $G$ as the product of a compact subgroup, an abelian subgroup and a nilpotent subgroup, the projection $G\to G, g=kan\mapsto (an)^{-1}$ is a Rota-Baxter operator. Here note that $AN$ is a solvable Lie subgroup of $G$.
\item
The same conclusions can be obtained for some other decompositions of Lie groups, such as the Langlands decomposition~\cite[VII.7]{K} $P = MAN$ of a parabolic subgroup $P$ of a reductive Lie group as the product of a semisimple subgroup, an abelian subgroup and a nilpotent subgroup. Note that $AN$ is a Lie subgroup of $P$.
\end{enumerate}
\mlabel{ex:sl}
\end{ex}

\subsection{Rota-Baxter Lie groups as integrations of Rota-Baxter Lie algebras}

Generalizing the fact that Lie groups are integrations of Lie algebras,  we now show that Rota-Baxter Lie groups serve as integrations of Rota-Baxter Lie algebras of weight 1.

Let $G$ be a Lie group and $e$ its identity. Let $\g=T_eG$ be the Lie algebra of  $G$ and let
$$\exp^{(\cdot)}:\mathfrak{g}\longrightarrow G$$
be the exponential map. Then the relation between the Lie bracket $[\cdot,\cdot]_\g$ and the Lie group multiplication is given by the following important formula:
\begin{equation}\mlabel{eq:expo}
  [u,v]_\g=\frac{d^2}{dt ds}\eval{t,s=0}\exp^{tu}\exp^{sv}\exp^{-tu},\quad \forall~ u,v\in\g.
\end{equation}

Now we are ready to give the main result in this section, which states that the differentiation of a Rota-Baxter Lie group is a Rota-Baxter Lie algebra of weight 1.
\begin{thm}
  Let $(G,\B)$ be a Rota-Baxter Lie group. Let $\mathfrak{g}=T_e G$ be the Lie algebra of $G$ and
\begin{equation}\mlabel{eq:tangent}
B=\B_{*e}:\mathfrak{g}\longrightarrow \mathfrak{g}
\end{equation}
  the tangent map of $\B$ at the identity $e$. Then  $(\mathfrak{g},B)$ is a Rota-Baxter Lie algebra of weight $1$.
\mlabel{main1}
\end{thm}
\begin{proof}
Observe from \eqref{RBgroup} that $\B(e)=e$. Since $B=\B_{*e}$ is the tangent map of $\B$ at $e$, we have the following relation for  sufficiently small $t$:
\begin{equation}\mlabel{eq:comm}
\frac{d}{dt}\eval{t=0}\B(\exp^{tu})=\frac{d}{dt}\eval{t=0}\exp^{tB(u)}=B(u), \quad \forall~ u\in\g.
\end{equation}
Now we check the identity
\[[B(u),B(v)]_\g=B([B(u),v]_\g +[u,B(v)]_\g+[u,v]_\g).\]
By \eqref{eq:f1}, \eqref{eq:f2}, \eqref{eq:expo}-\eqref{eq:comm}, and using the Leibniz rule, we obtain
\begin{eqnarray*}
[B(u),B(v)]_\g&=&\frac{d^2}{dt ds}\eval{t,s=0}\exp^{tB(u)}\exp^{sB(v)} \exp^{-tB(u)} \quad (Eq.~\meqref{eq:expo})
\\ &=&\frac{d^2}{dt ds}\eval{t,s=0} \B(\exp^{tu})\B(\exp^{sv})\B(\exp^{-tu})
\quad (B=\B_{*e})
\\ &=&\frac{d^2}{dt ds}\eval{t,s=0}\B(\exp^{tu})\B(\exp^{sv}\Ad_{\B(\exp^{sv})} \exp^{-tu})\quad (Eq.~\meqref{RBgroup})\\
&=&\frac{d^2}{dt ds}\eval{t,s=0}\B(\exp^{tu}(\Ad_{\B(\exp^{tu})} \exp^{sv})( \Ad_{\B(\exp^{tu})\B(\exp^{sv})}\exp^{-tu}))
\quad (Eq.~\meqref{RBgroup}) \\
&=&
\B_{*e}\left(\frac{d^2}{dt ds}\eval{t,s=0}\Ad_{\B(\exp^{tu})} \exp^{sv}+\frac{d^2}{dt ds}\eval{t,s=0}\Ad_{\B(\exp^{sv})}\exp^{-tu}\right.\\ &&\qquad +\left.\frac{d^2}{dt ds}\eval{t,s=0}\exp^{tu}\exp^{sv}\exp^{-tu}\right)
\\ &=&B([B(u),v]_\g-[B(v),u]_\g+[u,v]_\g).
\end{eqnarray*}
  Therefore,  $(\mathfrak{g},B)$ is a Rota-Baxter Lie algebra of weight $1$.
\end{proof}

\begin{ex}
The tangent map of the inverse map $(\cdot)^{-1}:G\longrightarrow G$ is $-\rm{id}:\g\longrightarrow \g$, which is naturally a Rota-Baxter operator of weight $1$ on the Lie algebra $\g$. Moreover, the tangent map of the Rota-Baxter operator $\B:G\to G, \B(g_+g_-)=g_-^{-1}$  in Lemma \ref{lem:rbpg}  is $-P_-:\g\to \g, -P_-(u_++u_-)=-u_-$ in Lemma \ref{lem:RBP}, which is a Rota-Baxter operator of weight $1$ on $\g$.
\end{ex}

\begin{rmk}\label{rmk:relation}
 In the Lie algebra case, it is straightforward to see that   $(\g,B)$ is a Rota-Baxter Lie algebra of weight 1 if and only if $(\g,-B)$ is a Rota-Baxter Lie algebra of weight $-1$. Now at the Lie group level, we can define
a {\bf Rota-Baxter Lie group of weight $-1$} to be a Lie group $G$ with a map $\frkC:G\longrightarrow G$ such that
\begin{eqnarray}\mlabel{RBgroup-1}
\frkC(g_1)\frkC(g_2)=\frkC((\Ad_{\frkC(g_1)} g_2)g_1),\qquad \forall~ g_1,g_2\in G.
\end{eqnarray}
 Let $\frkB$ be a Rota-Baxter operator on a Lie group $G$. Define $\frkC:G\to G$ by
	$$ \frkC(g):=\frkB(g^{-1}).$$
		In Eq.~\meqref{RBgroup}, replacing $g_1$ and $g_2$ by $g_1^{-1}$ and $g_2^{-1}$, we obtain
$$\frkB(g_1^{-1})\frkB(g_2^{-1}) =\frkB(g_1^{-1}\Ad_{\frkB(g_1^{-1})}g_2^{-1})
=\frkB\Big((\big(\Ad_{\frkB(g_1^{-1})}g_2\big) g_1)^{-1}\Big).$$
This gives
$$\frkC(g_1)\frkC(g_2)=\frkC\Big(\big(\Ad_{\frkC(g_1)}g_2\big)g_1\Big).
$$
Therefore, $\frkC$ is a Rota-Baxter operator of weight $-1$.

Under the assumption given in Lemma \ref{lem:RBPG}, let $\frkC$ be the projection to the first direct factor, i.e.
$$
\frkC(g_+g_-):=g_+.
$$
Then it is straightforward to deduce that $\frkC$ is a Rota-Baxter operator of weight $-1$.

Similar to the proof of Theorem \mref{main1}, one can show that the differentiation of a Rota-Baxter Lie group of weight $-1$ is a Rota-Baxter Lie algebra of weight $-1$. Thus the above discussion can be summarized into the following diagram:
  \begin{equation}
\begin{split}
\xymatrix{
\text{RB Lie alg. $(\g,B)$ of weight 1} \ar^{\text{additive inverse}}[rrr] \ar_{\text{integration}}[d] &&& \text{RB Lie alg. $(\g,-B)$ of weight $-1$} \ar^{\text{integration}}[d]\\
\text{RB Lie group $(G,\B)$ of weight 1} \ar^{\text{multiplicative inverse}}[rrr] &&& \text{RB Lie group $(G,\frkC)$ of weight $-1$.}
}
\end{split}
\mlabel{eq:wei}
\end{equation}
{\it For simplicity, we will only consider Rota-Baxter Lie groups of weight $1$ in the rest of the paper.}
\end{rmk}

Now we give further characterization of Rota-Baxter operators on Lie groups.

Let $G$ be a Lie group. It acts on itself by the adjoint action $\Ad: G\longrightarrow \Aut(G)$. Thus we obtain a new Lie group structure on $G\times G$  given by
\[(g_1,h_1)\cdot (g_2,h_2):=(g_1g_2,h_1\Ad_{g_1} h_2).\]
We denote this Lie group  by $G\triangleright G$.

Let  $\B:G\longrightarrow G$ be a smooth map and denote its graph in $G\times G$ by $\huaG_\B$, that is,
$$
\huaG_\B:=\{(\B(g),g)\,| \,g\in G\}.
$$

\begin{pro}
Let $G$ be a Lie group. Then a smooth map $\B:G\longrightarrow G$ is a Rota-Baxter operator if and only if the graph  $\huaG_\B$ is a Lie subgroup of $G\triangleright G$.
\end{pro}
\begin{proof}
For all $g,h\in G$, we have
\[(\B(g),g)\cdot (\B(h),h)=(\B(g) \B(h),g\Ad_{\B(g)} h),\]
which implies  that the multiplication is closed on the graph of $\B$ if and only if
\[\B(g) \B(h)=\B(g\Ad_{\B(g)} h),\]
that is, $\B$ is a Rota-Baxter operator on the Lie group $G$.
\end{proof}

For a Rota-Baxter Lie algebra $(\g,B)$ of weight $1$, the bracket
\begin{eqnarray}\label{Bbr}
[u,v]_B=[B(u),v]_\g+[u,B(v)]_\g+[u,v]_\g,\quad \forall~ u,v\in \g,
\end{eqnarray}
defines another Lie algebra structure on $\g$, which we will call the {\bf \desc} of $[\cdot,\cdot]_\g$, while the pair $((\g,[\cdot,\cdot]_\g),(\g,[\cdot,\cdot]_B)$ is called the double in~\cite{STS,STS2}. Further $B$ is also a Rota-Baxter operator on $(\g,[\cdot , \cdot]_B)$ and $B: (\g,[\cdot,\cdot]_B)\to \g$ is a homomorphism of Rota-Baxter Lie algebras. This property also holds for Rota-Baxter Lie groups.

\begin{pro}\mlabel{G*}
Let $(G,\B)$ be a Rota-Baxter Lie group.
\begin{enumerate}
\item
The pair $(G,*)$, with the multiplication
\begin{eqnarray}\label{*}
g_1* g_2:=g_1\Ad_{\B(g_1)} g_2, \qquad \forall~g_1, g_2\in G,
\end{eqnarray}
is also a Lie group, called the {\bf \desc Lie group} of the Rota-Baxter Lie group $(G,\B)$, whose Lie algebra is the \desc Lie algebra $(\g,[\cdot,\cdot]_B)$, where $B=\B_{*e}$ and $[\cdot,\cdot]_B$ is given in \meqref{Bbr}.
\mlabel{it:G*1}
\item
The operator $\B$ is a Rota-Baxter operator on the Lie group $(G,*)$.
\mlabel{it:G*2}
\item
The map $\B:(G,*)\to G$ is a homomorphism of Rota-Baxter Lie groups from $(G,*,\B)$ to $(G,\B)$.
\mlabel{it:G*3}
\end{enumerate}
\end{pro}

\begin{proof}
\meqref{it:G*1}
Let $e$ be the identity of the Lie group $G$. It is direct to see that $\B(e)=e$. Also
$g* e=e* g=g$ and hence $e$ is also the identity of the multiplication $*$. Further,
\begin{eqnarray*}
(g_1* g_2)* g_3=(g_1\Ad_{\B(g_1)} g_2)* g_3=
g_1\Ad_{\B(g_1)} g_2\Ad_{\B(g_1\Ad_{\B(g_1)} g_2)} g_3
\end{eqnarray*}
and
\begin{eqnarray*}
g_1* (g_2* g_3)=g_1\Ad_{\B(g_1)}(g_2\Ad_{\B(g_2)} g_3)=
g_1 \Ad_{\B(g_1)} g_2\Ad_{\B(g_1)\B(g_2)}g_3.
\end{eqnarray*}
Hence the associativity holds since $\B$ satisfies \eqref{RBgroup}.

We can also see that the inverse of $g$ for the multiplication $*$, denoted by $g^{\dag}$, is $\Ad_{\B(g)^{-1}} g^{-1}$.

By \eqref{eq:expo}, we have
\begin{eqnarray*}
[u,v]_B&=&\frac{d^2}{dt ds}\eval{t,s=0}\exp^{tu}*\exp^{sv}*\exp^{-tu}
\\ &=&\frac{d^2}{dt ds}\eval{t,s=0}\exp^{tu}*(\exp^{sv}\Ad_{\B(\exp^{sv})} \exp^{-tu})\\
&=&\frac{d^2}{dt ds}\eval{t,s=0}\exp^{tu}(\Ad_{\B(\exp^{tu})} \exp^{sv})( \Ad_{\B(\exp^{tu})\B(\exp^{sv})}\exp^{-tu})
\\ &=&\frac{d^2}{dt ds}\eval{t,s=0}\Ad_{\B(\exp^{tu})} \exp^{sv}+\frac{d^2}{dt ds}\eval{t,s=0}\Ad_{\B(\exp^{sv})}\exp^{-tu}+\frac{d^2}{dt ds}\eval{t,s=0}\Ad_{\exp^{tu}} \exp^{sv}\\ &=&[B(u),v]_\g-[B(v),u]_\g+[u,v]_\g,\end{eqnarray*}
which gives the desired identity.

\smallskip

\noindent
\meqref{it:G*2}
We shall check
\begin{eqnarray}\label{RB*}
\B(g)*\B(h)=\B(g*\B(g)*h*\B(g)^{\dag}),
\end{eqnarray}
where $\B(g)^{\dag}$ is the inverse of $\B(g)$ with respect to the multiplication $*$.
By the fact that
\[\B(g*h)=\B(g)\B(h),\] we have
\[e=\B(\B(g)*\B(g)^{\dag})=\B^2(g)\B(\B(g)^{\dag}),\]
which implies that
$\B^2(g) ^{-1}=\B(\B(g)^{\dag})$. Based on this
and the definition of $*$,  we have
\begin{eqnarray*}
\B(g)*\B(h)&=&\B(g)\B^2(g)\B(h)\B^2(g)^{-1}=\B(g*\B(g)*h*\B(g)^{\dag}).
\end{eqnarray*}
This is \eqref{RB*}.

\smallskip

\noindent
\meqref{it:G*3} It follows from the definition of Rota-Baxter Lie groups that $\B$ is a homomorphism of Lie groups. Further the compatibility of $\B$ with the Rota-Baxter operators is tautology: $\B\circ \B=\B\circ \B$.
\end{proof}

\begin{cor}
For the  Rota-Baxter operator $\B: G\to G$ given in Lemma \ref{lem:rbpg}, we have a new group structure on $G$ given by
\[g*h=(g_+g_-)*(h_+h_-)=g_+g_-g_-^{-1}h_+h_-g_-=g_+h_+h_-g_-,\qquad \forall~g,h\in G.\]
The Lie algebra of this Lie group is given by
\[[u,v]_B=[u_++u_-,v_++v_-]_B=-[u_-,v_++v_-]_\g-[u_++u_-,v_-]_\g+[u_++u_-,v_++v_-]_\g=[u_+,v_+]_\g-[u_-,v_-]_\g,\] which  corresponds to the Rota-Baxter operator of weight $1$ given in Lemma \ref{lem:RBP}.
\end{cor}

\subsection{Differential operators on Lie groups and crossed homomorphisms}

Recall that a differential operator (of weight $\lambda$) on a Lie algebra is a linear map $D:\g\to\g$ such that $$
D[u,v]_\g=[D(u),v]_\g+[u,D(v)]_\g+\lambda[D(u),D(v)]_\g.
$$
As in the case for associative algebras~\cite{GK}, the differential operator of weight $\lambda$ on a Lie algebra is the formal left inverse of the Rota-Baxter operator of weight $\lambda$ on a Lie algebra.
It is thus natural to look for a differential operator of weight $\pm 1$ on a Lie group that meets the following requirements:
%
\delete{We have also seen above that the integration of a Rota-Baxter operator of $\pm 1$ on a Lie algebra is a Rota-Baxter operator of weight $\pm 1$ on a Lie group. Thus it is natural to ask what a differential operator of weight $\pm 1$ on a Lie group should be.
From the above discussion, we expect that such an operator should have meet the following requirements:}
\begin{enumerate}
	\item As in the case of Lie algebras and associative algebras, the differential operator on Lie groups should be the formal inverse of the Rota-Baxter operator on Lie groups;
	\item The tangent map of the differential operator on Lie groups should be the differential operator on Lie algebras.
\end{enumerate}

We now show that these requirements are fulfilled by the notion of a crossed homomorphism or a 1-cocycle with non-abelian coefficients~\cite{Se}.

\begin{defi}
	Let $G$ be a group and let $\Gamma$ be a group with a group action by $G$, given by $\alpha\mapsto \alpha^x, \alpha\in \Gamma, x\in G$. A map $f: G\to \Gamma$ is called a {\bf crossed homomorphism} or a {\bf 1-cocycle} if
	\begin{equation} \notag
	f(xy)=f(x)f(y)^x, \quad \forall x, y \in G.
	\end{equation}
\end{defi}

Taking $\Gamma$ to be the group $G$ itself equipped with the conjugation action of $G$, we give the following notion.

\begin{defi}
A smooth map $\frkD$ on a Lie group $G$ is called a {\bf differential operator of weight $1$} (resp. $-1$) if
\begin{equation}
\mlabel{de:diffgp}
\frkD(gh)=\frkD(g)\Ad_g\frkD(h) \quad  \big(\text{resp. } \frkD(gh)=(\Ad_g\frkD(h))\frkD(g)\big), \quad \forall g, h\in G.
\end{equation}
Then $(G,\frkD)$ is called a {\bf differential Lie group of weight $1$} (resp. $-1$).
\end{defi}

We now show that this notion fulfills the above requirements and thus its name is justified.
\begin{thm}
\begin{enumerate}
	\item \label{it:diffwt1} The formal inverse of a Rota-Baxter operator of weight $\pm 1$ on Lie groups is the differential operator of $\pm 1$.
	\item \label{it:diffwt2}
	Let $(G,\frkD)$ be a differential Lie group of weight $\pm 1$. Let $\g$ be the Lie algebra of the Lie group $G$ and let $D=\frkD_{*e}$ be the tangent map of $\frkD$. Then  $(\g,D)$ is a differential Lie algebra of the same weight.
\end{enumerate}
	\mlabel{pro:diffwt}
\end{thm}
\begin{proof}
	\eqref{it:diffwt1}
Let $\frkD$ be a formal inverse of the Rota-Baxter operator $\B$ on a Lie group $G$:
$$ \B(g)\B(h)=\B(g\Ad_{\B(g)}h),\quad  \forall g, h\in G.$$
Replacing $\B(g)$ and $\B(h)$ by $x$ and $y$ respectively leads to replacing $g$ and $h$ by $\frkD(x)$ and $\frkD(y)$ respectively. Then applying $\frkD$ to both sides of the above equation gives
$$ \frkD(xy)=\frkD(x)\Ad_x \frkD(y),\quad  \forall x, y\in G.$$
Thus $\frkD$ is a differential operator of weight $1$. The same proof works for the case of weight $-1$.
		
	\smallskip
	
	\noindent
	\eqref{it:diffwt2}
	Let $\frkD: G\to G$ be a differential operator of weight $1$. Note that $\frkD(e)=e$.
	We claim that the tangent map of $\frkD$ at the identity $D=\frkD_{*e}:\g\to \g$ satisfies $$D[u,v]_\g=[D(u),v]_\g+[u,D(v)]_\g+[D(u),D(v)]_\g.$$ In fact, by $\frac{d}{dt}\eval{t=0}\frkD(\exp^{tu})=\frac{d}{dt}\eval{t=0}\exp^{tD(u)}=D(u)$, \eqref{eq:expo} and \eqref{de:diffgp}, we have
	\begin{eqnarray*}
	 D[u,v]_\g&=&\frac{d^2}{dtds}\eval{t,s=0}\frkD(\exp^{tu}\exp^{sv}\exp^{-tu})\\ &=&\frac{d^2}{dtds}\eval{t,s=0}   \frkD(\exp^{tu}\exp^{sv})\Ad_{\exp^{tu}\exp^{sv}}\frkD(\exp^{-tu})\\ &=&\frac{d^2}{dtds}\eval{t,s=0} \frkD(\exp^{tu})\Ad_{\exp^{tu}} \frkD(\exp^{sv}) \Ad_{\exp^{tu}\exp^{sv}}\frkD(\exp^{-tu})\\ &=& \frac{d^2}{dtds}\eval{t,s=0}  \Ad_{\exp^{sv}}\frkD(\exp^{-tu})+\frac{d^2}{dtds}\eval{t,s=0}  \Ad_{\exp^{tu}}\frkD(\exp^{sv})\\ &&+\frac{d^2}{dtds}\eval{t,s=0}   \frkD(\exp^{tu})\frkD(\exp^{sv})\frkD(\exp^{-tu})\\ &=& [D(u),v]_\g+[u,D(v)]_\g+[D(u),D(v)]_\g.
\end{eqnarray*}
The same proof applies to the case when the weight is $-1$.
\end{proof}

\section{Factorization theorem of  Rota-Baxter Lie groups}
\mlabel{sec:fact}
In this section, we prove a factorization theorem of Rota-Baxter Lie groups which yields the Global Factorization Theorem of Semenov-Tian-Shansky for Lie groups given in~\cite{STS,STS2}. The notion of Rota-Baxter operators on Lie groups allows us to adapt the approach in~\cite{STS,STS2} from Lie algebras to Lie groups. See~\cite{FRS} for a similar factorization on Poisson Lie groups.

Let $(G, \B)$ be a Rota-Baxter Lie group. Denote by $G_\B$ the Lie group $G$ with the new group structure  $g_1*g_2:=g_1\Ad_{\B(g_1)} g_2$.

\begin{pro}\mlabel{B-}
Let $(G,\B)$  be a Rota-Baxter Lie group. Define \[\B_+:G\to G,\qquad  \B_+(g)=g\B(g).\] Then   $\B_{+}$ is a Lie group homomorphism from $G_{\B}$ to $G$.
\end{pro}
\begin{proof}
We have
\begin{eqnarray*}
\B_+(g_1\Ad_{\B(g_1)}g_2)&=&g_1(\Ad_{\B(g_1)}g_2)\B(g_1\Ad_{\B(g_1)}g_2)\\ &=&g_1\B(g_1)g_2\B(g_1)^{-1}\B(g_1)\B(g_2)\\
&=&g_1\B(g_1)g_2\B(g_2)\\ &=&\B_+(g_1)\B_+(g_2).
\end{eqnarray*}
So $\B_{+}$ is  a Lie group homomorphism.
\end{proof}
Let $\B$ be a Rota-Baxter operator on $G$. Define four subsets of $G$ as follows:
\[G_{+}:=\im \B_+,\qquad G_{-}:=\im \B, \qquad K_{+}:=\ker \B,\qquad  K_{-}:=\ker \B_{+}.\]
Since both $\B$ and $\B_+$ are Lie group homomorphisms,  $G_+$ and $G_-$ are Lie subgroups of $G$, $K_+$ and $K_-$ are normal Lie subgroups of $G_\B$, and $G_{\pm}\cong G_{\B}/K_{\mp}$. Moreover, we have the following relations.

\begin{lem} \mlabel{lem:normal}
$K_+\subset G_+$ and $K_-\subset G_-$ are normal Lie subgroups.
\end{lem}
\begin{proof}
Let $k\in K_-$, that is, $\B_+(k)=k\B(k)=e$. Then we have $k=\B(k)^{-1}=\B(k^{\dag})$ by \eqref{RBgroup}, where  $k^{\dag}$ is the inverse of $k$ in $G_\B$. Thus $k\in G_-$ and $K_-\subset G_-$.

For any $k\in K_-$ and $\B(g)\in G_-$, let us check $\B(g)k \B(g)^{-1}\in K_-$.  If fact, note that $g^\dag=\Ad_{\B(g)^{-1}}g^{-1}$. Then we  have
\begin{eqnarray*}
g*k*g^\dag&=&g\Ad_{\B(g)} k\Ad_{\B(g)\B(k)} \Ad_{\B(g)^{-1}}g^{-1}\\ &=&g\B(g)k\B(g)^{-1}\B(g)\B(k)\B(g)^{-1}g^{-1}\B(g)\B(k)^{-1}\B(g)^{-1}\\ &=&\B(g)k\B(g)^{-1},
\end{eqnarray*}
where in the last equation we used the fact that $k\B(k)=e$.
Thus,
\[\B(g)k\B(g)^{-1}\B(\B(g)k\B(g)^{-1})=\B(g)k\B(g)^{-1} \B(g*k*g^{\dag})=\B(g)k\B(g)^{-1}\B(g)\B(k)\B(g)^{-1}=e,\]
that is, $\B(g)k \B(g)^{-1}$ is in $K_-$. Thus  $K_-\subset G_-$ is a normal Lie subgroup.

Likewise we can prove that $K_+$ is a normal Lie subgroup of $G_+$.
\end{proof}

Based on Lemma~\mref{lem:normal}, we define
a map
\[\Theta: G_-/K_-\to G_{+}/K_{+}, \quad \Theta(\overline{\B(g)})=\overline{\B_{+}(g)},\quad \forall~ g\in G,\]
 where $\overline{\cdot}$ denotes the equivalence class in the two quotients. To see that $\Theta$ is well-defined, let $k\in K_-$. We have
 $k=\B(k)^{-1}=\B(k^\dag)$. Note that $k^\dag=\Ad_{\B(k)^{-1}}k^{-1}=k^{-1}$.
 We have
\begin{eqnarray*}
\Theta(\overline{\B(g)k})&=&\Theta(\overline{\B(g* k^{-1})})=\overline{\B_+(g*k^{-1})}\\ &=&\overline{g*k^{-1} \B(g*k^{-1})}\\
&=&\overline{g\B(g)k^{-1}\B(g)^{-1}\B(g)\B(k^{-1})}\\ &=&\overline{g\B(g)}=\Theta(\overline{\B(g)}),
\end{eqnarray*}
which implies that $\Theta$ is well-defined.
\begin{pro}\mlabel{CB}
The map $\Theta: G_-/K_-\to G_{+}/K_{+}$ is a Lie group isomorphism, and is called the {\bf Cayley transform} of the Rota-Baxter operator $\B$.
\end{pro}
The term Cayley transform is adapted from a similar notion from $r$-matrices~\cite{ES,STS}.
\begin{proof}
It is obvious that $\Theta$ is surjective. To see that it is also injective, if $\B_+(g)=g\B(g)\in K_+$, that is, $\B(g\B(g))=e$, then
we have
\[\B_+(\B(g))=\B(g)\B(\B(g))=\B(g*\B(g))=\B(g\Ad_{\B(g)} \B(g))=\B(g\B(g))=e,\]
which implies that $\B(g)\in K_-$. This proves that $\Theta$ is injective.

We next show that $\Theta$ is a Lie group homomorphism, which follows from
\[\Theta(\overline{\B(g)}\overline{\B(g')})=\Theta(\overline{\B(g*g')})=\overline{\B_+(g*g')}=\overline{\B_+(g)}\overline{\B_+(g')}=\Theta(\overline{\B(g)})\Theta(\overline{\B(g)}),\]
by Proposition \mref{B-}. Therefore, $\Theta$ is a Lie group isomorphism.
\end{proof}
Now we consider the product Lie group $(G_+\times G_-,\cdot_D)$, whose Lie group product is
\[(g_+,g_-)\cdot_D(h_+,h_-):=(g_+h_+,g_-h_-), \qquad \forall~g_+,h_+\in G_+,g_-,h_-\in G_-.\]
Let $G_\Theta\subset G_+\times G_-$ denote the subset
\[G_\Theta:=\left\{(g_+,g_{-})\in G_+\times G_-; \Theta(\overline{g_-})=\overline{g_+}\right\}.\]
Define a map $\Phi:G\to G_\Theta$ by
\[\Phi(g):=(\B_+(g),\B(g)).\]
\begin{lem}
With the above notations, $G_\Theta$ is a Lie subgroup of $(G_+\times G_-,\cdot_D)$. Moreover, the map $\Phi$ is a Lie group isomorphism from $(G,*)$ to $G_\Theta$.
\end{lem}
\begin{proof}
By Proposition \mref{CB}, for any $(g_+,g_-),(h_+,h_-)\in G_\Theta$, we have
\[\Theta(\overline{g_-h_-})=\Theta(\overline{g_-}\overline{h_-})=\Theta(\overline{g_-})\Theta(\overline{h_-})=\overline{g_+}\overline{h_+}=\overline{g_+h_+},\]
which implies that $(g_+h_+,g_-h_-)\in G_\Theta$ and then $G_\Theta$ is a subgroup of $(G_+\times G_-,\cdot_D)$.

We next check that $\Phi$ is a bijection. Let $g\in G$ such that $\Phi(g)=(e,e)$. Then we have $\B(g)=g\B(g)=e$. Thus $g=e$, which implies that $\Phi$ is injective. For any $(g_+,g_-)\in G_\Theta$, we have $\Theta(\overline{g_-})=\overline{g_+}$. Since $g_-\in G_-$, there exists $g\in G$ such that $\B(g)=g_-$. Hence we get
\[\Theta(\overline{g_-})=\Theta(\overline{\B(g)})=\overline{g\B(g)}.\]
Therefore, $\overline{g_+}=\overline{g\B(g)}$, which means that there exists $k\in K_+$ such that
\[g_+=g\B(g)k.\]
Let $g'=g*k$. Then we have \[\Phi(g')=(g*k\B(g*k),\B(g*k))=(g\B(g)k\B(g)^{-1}\B(g),\B(g))=(g_+,g_-).\]Therefore, $\Phi$ is surjective.

Finally, for any $g,h\in G$, by Proposition \mref{B-}, we have
\begin{eqnarray*}
\Phi(g*h)&=&(\B_+(g*h),\B(g*h))\\ &=&(\B_+(g)\B_+(h),\B(g)\B(h))\\ &=&(\B_+(g),\B(g))\cdot_D(\B_+(h),\B(h))\\ &=&\Phi(g)\cdot_D\Phi(h).
\end{eqnarray*}
Therefore, $\Phi$ is a Lie group homomorphism and hence a Lie group isomorphism.
\end{proof}

\begin{thm}\label{thm:facgroup}
\rm{(Factorization theorem of Rota-Baxter Lie groups)}
Let $(G,\B)$ be a Rota-Baxter Lie group. Then  every element $g\in G$ can be uniquely expressed as $g=g_+g_{-}^{-1}$ for $(g_+,g_-)\in G_\Theta$.
\mlabel{thm:factgp}
\end{thm}
\begin{proof}
For any $g\in G$, we have $g=\B_+(g)\B(g)^{-1}$. To see the uniqueness, if $g=g_+g_-^{-1}=h_+h_-^{-1}$, then $h_+^{-1}g_+=h_-^{-1}g_-\in G_+\cap G_-$ and $\Theta(\overline{h_-^{-1}g_-})=\overline{h_+^{-1}g_+}$.  Suppose $h_+^{-1}g_+=h_-^{-1}g_-=\B(s)\in G_+\cap G_-$ for some $s\in G$. Then
\[\Theta(\overline{\B(s)})=\overline{\B_+(s)}=\overline{s\B(s)}=\overline{\B(s)},\]
 which implies that $s\B(s)=\B(s)k$ for some $k\in K_+$. Note that $K_+\subset G_+$ is a normal subgroup. We have $s=\B(s)k\B(s)^{-1}\in K_+$. Therefore, $\B(s)=e$. Hence we get that $h_+=g_+$ and $h_-=g_-$.
\end{proof}

Conversely, we obtain the following generalization of Lemma~\ref{lem:RBPG} by the same argument.
\begin{pro}
Suppose  that $G$ is a Lie group together with two Lie subgroups $G_\pm$. Let $\tilde{G}\subset G_+\times G_-$ be a Lie subgroup. Assume that each element $g\in G$ can be uniquely decomposed as
\begin{eqnarray}\mlabel{fp}
g=g_+g_-^{-1}, \qquad (g_+,g_-)\in \tilde{G}.
\end{eqnarray}
Then $\B:G\to G$ defined by $\B(g_+g_-^{-1})=g_-$ is a Rota-Baxter operator on the Lie group $G$.
\end{pro}

At the end of this section, we give some remarks on the applications of the factorization theorem of Rota-Baxter Lie groups to the factorization theorems of Semenov-Tian-Shansky~\cite{STS,STS2}. We refer the reader to the original references for the related background.

\begin{rmk} \label{rk:rlnsts}
  \begin{enumerate}
    \item \label{it:rlnsts1}
    Applying differentiation, Theorem \ref{thm:facgroup} gives a factorization of Rota-Baxter Lie algebras of weight 1, whose form is similar to the Atkinson factorization of Rota-Baxter associative algebras~\cite{At,Gub}. As noted in the introduction, $B$ is a Rota-Baxter operator of weight 1 on a Lie algebra $(\g,[\cdot,\cdot]_\g)$ if and only if $R:=\id +2B$ satisfies the modified Yang-Baxter equation in \eqref{eq:mybe}.
Then it is straightforward to see that the above factorization of Rota-Baxter Lie algebras of weight $1$ coincides with the Infinitesimal Factorization Theorem of Semenov-Tian-Shansky for Lie algebras~\cite[Prop.~9]{STS}.
\item   Integrating   the above Infinitesimal Factorization Theorem of Lie algebras, Semenov-Tian-Shansky obtained his famous Global Factorization Theorem for Lie groups, held in a small enough neighborhood of the unit. This is a fundamental tool in studying integrable systems. See \cite[Theorem.~11]{STS} and \cite[Theorem.~3.3]{STS2} for more details. Our Theorem~\ref{thm:facgroup} shows that the Global Factorization Theorem can be derived directly on the group level and is a truly global result since it applies to all elements of the Lie group.
\end{enumerate}
\end{rmk}

In summary, we have the following commutative diagram on the relationship between the infinitesimal and global versions of the Factorization Theorem, where the classical approach took the right-down path and our approach takes the down-right path.
  \begin{equation}
\begin{split}
\xymatrix{
	(\g,R) \ar^{\text{factorization}}[rrr]\ar_{R=\id+2B}[d]&&& \g_\theta\subseteq \g_+\oplus \g_- \ar@{=}[d]\\
	(\g,B) \ar^{\text{factorization}}[rrr]  \ar_{\text{integration}}[d] &&& \g_\theta\subseteq \g_+\oplus \g_-  \ar^{\text{integration}}[d]\\
	(G,\B) \ar^{\text{factorization}}[rrr] &&& G_\Theta\subseteq G_+\times G_-.&&&
}
\end{split}
\mlabel{eq:factorcompare}
\end{equation}
Here the factorizations in the first two rows are the Infinitesimal Factorization Theorem for $(\g,R)$ from the modified Yang-Baxter equation by Semenov-Tian-Shansky and its equivalent form for the Rota-Baxter Lie algebra $(\g,B)$ in Remark~\ref{rk:rlnsts}\eqref{it:rlnsts1}.

\section{Rota-Baxter Lie algebroids and Rota-Baxter Lie groupoids}
\mlabel{sec:oid}
In this section, we introduce the notions of Rota-Baxter Lie algebroids and Rota-Baxter Lie groupoids. We show that a Rota-Baxter Lie algebroid gives rise to a post-Lie algebroid and that a Rota-Baxter Lie algebroid can be obtained from a Rota-Baxter groupoid by differentiation. Moreover, actions of Rota-Baxter Lie algebras, actions of Rota-Baxter Lie groups and actions of post-Lie algebras are introduced to produce Rota-Baxter Lie algebroids, Rota-Baxter Lie groupoids and post-Lie algebroids.

\subsection{Rota-Baxter Lie algebroids}

We introduce the notion of a Rota-Baxter Lie algebroid and show that an action of a Rota-Baxter Lie algebra naturally gives rise to a  Rota-Baxter Lie algebroid.

\begin{defi} (\mcite{Mkz:GTGA})
  A {\bf Lie algebroid} structure on a vector bundle $\huaA\longrightarrow M$ is
a pair consisting of a Lie algebra structure $[\cdot,\cdot]_\huaA$ on
the section space $\Gamma(\huaA)$ and a vector bundle morphism
$$a_\huaA:\huaA\longrightarrow TM$$
from $\huaA$ to the tangent bundle $TM$, called the {\bf anchor}, satisfying the relation
\begin{equation}\mlabel{eq:LAf}~[x,fy]_\huaA=f[x,y]_\huaA+a_\huaA(x)(f)y,\quad \forall~x,y\in\Gamma(\huaA),~f\in
\CWM.\end{equation}
When $a_\huaA$ is surjective, we call $\huaA$ a {\bf transitive Lie algebroid}.
\end{defi}
We usually denote a Lie algebroid by $(\huaA,[\cdot,\cdot]_\huaA,a_\huaA)$ or simply $\huaA$
if there is no danger of confusion.

We now introduce the notion of a Rota-Baxter operator on a Lie algebroid.
\begin{defi}
Given a scalar $\lambda$, a {\bf Rota-Baxter operator of weight $\lambda$} on a transitive Lie algebroid $(\huaA,[\cdot,\cdot]_\huaA,a_\huaA)$ is a bundle map $\huaB:\ker(a_\huaA)\longrightarrow  \huaA$  covering the identity  such that
\begin{eqnarray}\mlabel{eq:RB1}
\qquad [\huaB(u),\huaB(v)]_\huaA=\huaB([\huaB(u),v]_\huaA)+\huaB([u,\huaB(v)]_\huaA)+\lambda \huaB([u,v]_\huaA),\qquad \forall~ u,v\in \Gamma(\ker(a_\huaA)).
\end{eqnarray}
A {\bf Rota-Baxter Lie algebroid} is a Lie algebroid equipped with a  Rota-Baxter operator of weight $1$.
\end{defi}

Since $\ker (a_\huaA)$ is an ideal of $\huaA$, the terms in \eqref{eq:RB1} are well-defined.

\begin{rmk}
Here we use $\ker (a_\huaA)$ instead of $\huaA$ to ensure that Eq.~\eqref{eq:RB1} is compatible with the function linearity of $\huaB$, that is,  $\huaB(fu)=f\huaB(u)$ for any function $f\in \CWM$. Indeed, imposing the function linearity for a bundle map $\huaB:\huaA\lon \huaA$ in the definition of a  Rota-Baxter operator of weight $\lambda$ on a   Lie algebroid $(\huaA,[\cdot,\cdot]_\huaA,a_\huaA)$, we obtain
\begin{eqnarray*}
 0&=&[\huaB(fu),\huaB(v)]_\huaA-\huaB([\huaB(fu),v]_\huaA+[fu,\huaB(v)]_\huaA+\lambda [fu,v]_\huaA)\\
 &=&[f\huaB(u),\huaB(v)]_\huaA-\huaB([f\huaB(u),v]_\huaA+[fu,\huaB(v)]_\huaA+\lambda [fu,v]_\huaA)\\
 &=&f[\huaB(u),\huaB(v)]_\huaA-a_\huaA(\huaB(v))(f)\huaB(u)-f\huaB[\huaB(u),v]_\huaA+a_\huaA(v)(f)\huaB^2(u)\\
 &&-f\huaB[u,\huaB(v)]_\huaA+a_\huaA(\huaB(v))(f)\huaB(u)-\lambda f\huaB [u,v]_\huaA+\lambda a_\huaA(v)(f)\huaB(u)\\
 &=&a_\huaA(v)(f)\huaB^2(u)+\lambda a_\huaA(v)(f)\huaB(u).
\end{eqnarray*}
Therefore, $\huaB$ needs to satisfy the additional condition
$$
a_\huaA(v)(f)\huaB^2(u)+\lambda a_\huaA(v)(f)\huaB(u)=0,\quad \forall u,v\in \Gamma(\huaA), f\in \CWM.
$$
Thus, it is natural to restrict the domain of $\huaB$ to $\ker (a_\huaA)$.
\end{rmk}

\begin{rmk}
A vector space is a vector bundle over a point. Therefore, a Lie algebra  is naturally a Lie algebroid with the anchor being zero.   It is obvious that a Rota-Baxter Lie algebroid reduces to a Rota-Baxter Lie algebra when the underlying Lie algebroid reduces to a Lie algebra.
\end{rmk}

We first give some simple examples of Rota-Baxter Lie algebroids.
\begin{ex}
  Let $(\g,[\cdot,\cdot]_\g,B)$ be a Rota-Baxter Lie algebra of weight $1$, and $M$ a manifold. Consider the trivial bundle $M\times\g$. Then the linear map $B:\g\lon\g$ naturally gives rise to a bundle map $\huaB:M\times\g\lon M\times\g,~(m,u)\longmapsto (m,B(u))$. Furthermore,  the Lie bracket $[\cdot,\cdot]_\g$ can be naturally extended to $\Gamma(M\times \g)=C^\infty(M)\otimes \g$, which is function linear.    Then it is obvious that $(M\times\g, [\cdot,\cdot]_\g,\huaB) $ is a Rota-Baxter Lie algebroid with the anchor being zero.
\end{ex}

\begin{ex}\label{ex:minus}
Let $(\huaA,[\cdot,\cdot]_\huaA,a_\huaA)$ be a transitive Lie algebroid. The minus of the inclusion
\[\huaB:\ker (a_\huaA)\mapsto \huaA,\qquad \huaB(u)=-u,\]
is naturally a Rota-Baxter operator of weight $1$ on $\huaA$.
\end{ex}

Now we construct a class of examples of particular interest, namely the action Rota-Baxter Lie algebroids.

First we recall actions of Lie algebras on manifolds and the associated action Lie algebroids (\mcite{Mkz:GTGA}). Let $\phi: \mathfrak{g}\longrightarrow \mathfrak{X}(M)$ be a left action of a Lie algebra  $\mathfrak{g}$   on a manifold $M$, that is, a Lie algebra homomorphism from $(\g,[\cdot,\cdot]_\g)$ to the Lie algebra $(\frkX(M),[\cdot,\cdot]_{TM})$ of vector fields. Then we have a Lie algebroid structure on the trivial bundle $\huaA=M\times\g$,
whose anchor $a_\huaA:M\times\g\lon TM$ and  Lie bracket $[\cdot,\cdot]_\huaA:\wedge^2(C^\infty(M)\otimes \g)\lon C^\infty(M)\otimes \g$ are given by
\begin{eqnarray}\mlabel{action1}
a_\huaA(m,u)&=&\phi(u)_m,\quad \forall m\in M,~u\in\g,\\  \mlabel{action2}{[fu,gv]}_\huaA&=&fg[u,v]_{\mathfrak{g}}+f\phi(u)(g)v-g\phi(v)(f)u,\quad \forall~u,v\in \mathfrak{g},f,g\in C^\infty(M).
\end{eqnarray}
This Lie algebroid is called the {\bf action Lie algebroid} of the Lie algebra $(\g,[\cdot,\cdot]_\g)$ and the action $\phi$, and denoted by $\g\times_\phi M$. If $\phi$ is a transitive action, then  $\g\times_\phi M$ is a transitive Lie algebroid.

\begin{pro}\mlabel{actioncase}
With the above notations,
a bundle map $\huaB:\ker (a_\huaA)\to \g\times_\phi M$ gives a Rota-Baxter operator of weight $1$ on $\g\times_\phi M$  if and only if $\huaB_m:=\huaB|_m:\ker (a_\huaA)_m\to \g$  for any $m\in M$ satisfies that
\[[\huaB_m(u),\huaB_m(v)]_\g=\huaB_m([\huaB_m(u),v]_\g+[u,\huaB_m(v)]_\g+[u,v]_\g),\qquad \forall~u,v\in \ker (a_\huaA)_m.\]
\end{pro}

\begin{proof}
  It follows from a straightforward verification.
\end{proof}

\begin{defi}
  An {\bf action of a Rota-Baxter Lie algebra} $(\g,[\cdot,\cdot]_\g,B)$ on a manifold $M$ is a homomorphism of Lie algebras $\phi:(\g,[\cdot,\cdot]_B)\lon \frkX(M)$,  where   the Lie bracket $[\cdot,\cdot]_{B}$ is defined by \eqref{Bbr}.
\end{defi}

Let $\phi:(\g,[\cdot,\cdot]_B)\lon \frkX(M)$ be an action of the Rota-Baxter Lie algebra  $(\g,[\cdot,\cdot]_\g,B)$ on a manifold $M$. Consider the direct sum bundle $\huaA:=(M\times \g)\oplus TM$. Then $\Gamma(\huaA)=(C^\infty(M)\otimes \g) \oplus \mathfrak{X}(M)$.  There  is naturally a Lie algebroid structure on $\huaA$ whose anchor $a_\huaA$ is the projection to $TM$ and whose bracket is determined by
 $$
  [fu+X,gv+Y]_\huaA:=fg[u,v]_\g+X(g) v-Y(f)u+[X,Y]_{TM},
  $$
for all $~X,Y\in \mathfrak{X}(M),u,v\in \g, f,g\in C^\infty(M)$.

Consider the bundle map $\huaB: \ker (a_\huaA)=M\times \g\longrightarrow (M\times \g)\oplus TM$  defined by
\begin{equation}\label{eq:actionRB}
\huaB(m,u):=(m,B(u), \phi(u)(m)),\qquad \forall m\in M,u\in \g.
\end{equation}

 \begin{pro}\label{pro:actionRB}
 With the above notations,
the bundle map $\huaB$ defined by \eqref{eq:actionRB} is a Rota-Baxter operator on the Lie algebroid $((M\times \g)\oplus TM,[\cdot,\cdot]_\huaA,a_\huaA)$.
 \end{pro}
 The Rota-Baxter algebroid $((M\times \g)\oplus TM,[\cdot,\cdot]_\huaA,a_\huaA,\huaB)$ will be called the {\bf action Rota-Baxter algebroid} in the sequel.
 \begin{proof}
   The bundle map $\huaB$
is a Rota-Baxter operator of weight $1$ on $\huaA$, that is, \eqref{eq:RB1} holds for $\lambda=1$, if and only if
\begin{eqnarray*}
[B(u),B(v)]_\g&=&B([B(u),v]_\g+[u,B(v)]_\g+[u,v]_\g),\qquad \forall~ u,v\in \g,\\
{}[\phi(u),\phi(v)]_{TM}&=&\phi([B(u),v]_\g+[u,B(v)]_\g+[u,v]_\g).
\end{eqnarray*}
That is, $B$ is a Rota-Baxter operator of weight $1$ on the Lie algebra $(\g,[\cdot,\cdot]_\g)$ and $\phi$ is a left action of the Lie algebra $(\g, [\cdot,\cdot]_{B})$ on $M$.
 \end{proof}

\subsection{Relations to post-Lie algebroids}

A notion that is closely related to a Rota-Baxter Lie algebra is the post-Lie algebra. We recall some background in order to present a more general relationship between Rota-Baxter Lie algebroids and post-Lie algebroids. Moreover, we introduce the notion of actions of post-Lie algebras on manifolds, which produce a class of interesting  post-Lie algebroids.

\begin{defi} (\mcite{Val})
A {\bf post-Lie algebra} $(\g,[\cdot,\cdot]_\g,\rhd)$ consists of a Lie algebra $(\g,[\cdot,\cdot]_\g)$ and a binary product $\rhd:\g\otimes\g\lon\g$ such that
\begin{eqnarray}
\mlabel{Post-1}u\rhd[v,w]_\g&=&[u\rhd v,w]_\g+[v,u\rhd w]_\g,\\
\mlabel{Post-2}[u,v]_\g\rhd w&=&a_\rhd(u,v,w)-a_\rhd(v,u,w),
\end{eqnarray}
here $a_{\rhd}(u,v,w):=u\rhd(v\rhd w)-(u\rhd v)\rhd w $ and $u,v,w\in \g.$
\end{defi}

Define $L_\rhd:\g\lon \gl(\g)$ by $L_\rhd(u)(v)=u\rhd v$. Then by \eqref{Post-1},    $L_\rhd$ is a linear map from $\g$ to $\Der(\g)$.
\begin{rmk}
Let $(\g,[\cdot,\cdot]_\g,\rhd)$ be a post-Lie algebra. If the Lie bracket $[\cdot,\cdot]_\g=0$, then $(\g,\rhd)$ becomes a pre-Lie algebra. Thus,  a post-Lie algebra can be viewed as a nonabelian version of a pre-Lie algebra. See \mcite{BG,Burde16,Burde19} for the classifications of post-Lie algebras on certain Lie algebras, and \mcite{BGN,Munthe-Kaas-Lundervold} for applications of post-Lie algebras in integrable systems and numerical integrations.
\end{rmk}

The notion of a post-Lie algebroid was given in \mcite{Munthe-Kaas-Lundervold} in the study of geometric numerical analysis. Post-Lie algebroids are geometrization of post-Lie algebras. See \cite{MSV} for more applications of post-Lie algebroids.

\begin{defi}\mlabel{defi:postLA}
A {\bf post-Lie algebroid} structure on a vector bundle
$A\longrightarrow M$ is a triple that consists of a $C^\infty(M)$-linear Lie
algebra structure $[\cdot,\cdot]_A$ on $\Gamma(A)$, a  bilinear operation   $\rhd_A:\Gamma(A)\times \Gamma(A)\longrightarrow \Gamma(A)$   and a
vector bundle morphism $a_A:A\longrightarrow TM$, called the {\bf anchor},
such that $(\Gamma(A),[\cdot,\cdot]_A,\rhd_A)$ is a post-Lie algebra, and
for all $f\in\CWM$ and $u,v\in\Gamma(A)$, the following
relations are satisfied:
\begin{itemize}
\item[\rm(i)]$~u\rhd_A(fv)=f(u\rhd_A v)+a_A(u)(f)v,$
\item[\rm(ii)] $(fu)\rhd_A v=f(u\rhd_A v).$
\end{itemize}
\end{defi}

We usually denote a post-Lie algebroid by $(A,[\cdot,\cdot]_A,\rhd_A, a_A)$.  If the Lie algebra structure $[\cdot,\cdot]_A$ in a post-Lie algebroid $(A, [\cdot,\cdot]_A, \rhd_A, a_A)$ is abelian, then it becomes a left-symmetric algebroid, which is also called a Koszul-Vinberg algebroid. See \mcite{LSB2,LSBC,Boyom1,Boyom2} for more details.

A Rota-Baxter Lie algebra gives rise to a post-Lie algebra~\cite{BBGN}. More precisely, let $(\g,[\cdot,\cdot]_\g,B)$ be a Rota-Baxter Lie algebra. Then for $\rhd$ defined by
$$u\rhd v=[B(u),v]_\g,$$
 $(\g,[\cdot,\cdot]_\g,\rhd)$ is a post-Lie algebra.
We call it the {\bf splitting post-Lie algebra} of the Rota-Baxter Lie algebra $(\g,B)$ since this is a special case of the process of the splitting of an operad with a Rota-Baxter operator~\cite{BBGN}.
As the geometrization of this fact, we have

\begin{thm} \mlabel{thm:rbpost}
Let  $(\huaA,[\cdot,\cdot]_\huaA,a_\huaA,\huaB)$ be  a Rota-Baxter   Lie algebroid. Let $A=\ker(a_\huaA).$ Then for $\rhd_A$ defined by
\begin{eqnarray}
u\rhd_A v:=[\huaB(u),v]_\huaA,\quad \forall u,v\in\Gamma(\ker(a_\huaA)),
\end{eqnarray}
$(A,[\cdot,\cdot]_\huaA,\rhd_A, a_\huaA\circ \huaB)$ is a post-Lie algebroid, called the {\bf splitting post-Lie algebroid} of the Rota-Baxter Lie algebroid $(\huaA,[\cdot,\cdot]_\huaA,a_\huaA,\huaB)$.
 \end{thm}
 \begin{proof}
   By \eqref{eq:LAf}, the operation $[\cdot,\cdot]_\huaA$ on $\Gamma(\ker(a_\huaA))$ is $C^\infty(M)$-linear. By \eqref{eq:RB1}, we can deduce that $(\ker(a_\huaA),[\cdot,\cdot]_\huaA,\rhd_A)$ is a post-Lie algebra. Finally, for all $u,v\in\Gamma(\ker(a_\huaA))$ and $f\in C^\infty(M)$, we have
   \begin{eqnarray*}
     (fu)\rhd_A v&=&[f\huaB(u),v]_\huaA=f[\huaB(u),v]_\huaA=f u\rhd_A v,\\
     u\rhd_A (fv)&=&[\huaB(u),fv]_\huaA=f[\huaB(u),v]_\huaA+a_\huaA\circ \huaB(u)(f)v=f u\rhd_A v+a_\huaA\circ \huaB(u)(f)v,
   \end{eqnarray*}
   which implies that $(\ker(a_\huaA),[\cdot,\cdot]_\huaA,\rhd_A, a_\huaA\circ \huaB)$ is a post-Lie algebroid.
 \end{proof}

\begin{ex}
  Let $(\g,[\cdot,\cdot]_\g,\rhd)$ be a post-Lie algebra. Consider the trivial bundle $A=M\times \g.$ Then the operation $\rhd$  can be naturally extended to $\rhd_A$ defined on $\Gamma(M\times \g)=C^\infty(M)\otimes \g$, which is function linear.  The Lie bracket $[\cdot,\cdot]_\g$ can be similarly extended. It is obvious that $(M\times \g,[\cdot,\cdot]_\g,\rhd_A)$ is naturally a post-Lie algebroid with the anchor map being zero. We call it the {\bf post-Lie algebra  bundle}.
\mlabel{ex:bundlepL}
\end{ex}

\begin{ex}
The induced post-Lie algebroid of the Rota-Baxter operator given in Example \ref{ex:minus} is $(\ker (a_\huaA),[\cdot,\cdot]_\huaA,\rhd_A,0)$, where
$$u\rhd_A v:=-[u,v]_\huaA,\quad\forall  u,v\in\Gamma(\ker(a_\huaA)).$$
\end{ex}

In the sequel, we introduce actions of post-Lie algebras on manifolds, which will produce a class of interesting examples of post-Lie algebroids. Recall that a  {post-Lie algebra} $(\g,[\cdot,\cdot]_\g,\rhd)$ gives rise to a new {\bf subjacent} Lie bracket $\Courant{\cdot,\cdot}$:
\begin{equation}\label{eq:courant}
\Courant{u,v}=[u,v]_\g+u\rhd v-v\rhd u.
\end{equation}
\begin{defi}
  An {\bf action of a post-Lie algebra} $(\g,[\cdot,\cdot]_\g,\rhd)$ on a manifold $M$ is homomorphism of Lie algebras $\phi:(\g,\Courant{\cdot,\cdot})\lon \frkX(M).$
\end{defi}

\begin{rmk}
  If the post-Lie algebra reduces to a pre-Lie algebra, then the above definition reduces to the definition of actions of pre-Lie algebras on manifolds, which was given in~\cite{LSBC} for the study of left-symmetric algebroids.
\end{rmk}

It is straightforward to see that an action of a Rota-Baxter Lie algebra also gives an action of the splitting post-Lie algebra.

\begin{lem}\label{lem:repRBPL}
  Let $\phi:(\g,[\cdot,\cdot]_B)\lon \frkX(M)$ be an action of the Rota-Baxter Lie algebra  $(\g,[\cdot,\cdot]_\g,B)$ on a manifold $M$. Then $\phi$ is also an action of the splitting post-Lie algebra $(\g,[\cdot,\cdot]_\g,\rhd)$ from the Rota-Baxter Lie algebra.
\end{lem}

Let $\phi:(\g,\Courant{\cdot,\cdot})\lon \frkX(M)$ be an {action of the post-Lie algebra} $(\g,[\cdot,\cdot]_\g,\rhd)$ on a manifold $M$. On the trivial bundle $A=M\times \g$, define the anchor $a_A:M\times\g\lon TM$ and  the bilinear operation $\rhd_A:\otimes^2(C^\infty(M)\otimes \g)\lon C^\infty(M)\otimes \g$   by
\begin{eqnarray}\mlabel{action11}
a_A(m,u)&:=&\phi(u)_m,\quad \forall m\in M,~u\in\g,\\
 \mlabel{action21}  (fu)\rhd_A gv  &:=&fgu\rhd v+f\phi(u)(g)v,\quad \forall~u,v\in \mathfrak{g},f,g\in C^\infty(M).
\end{eqnarray}

\begin{pro}\label{pro:actionPL}
  With the above notations, $(M\times\g, [\cdot,\cdot]_\g,\rhd_A,a_A)$ is a post-Lie algebroid, called the {\bf action post-Lie algebroid} of the post-Lie algebra $(\g,[\cdot,\cdot]_\g,\rhd)$.
\end{pro}
\begin{proof}
  It follows from a direct verification.
\end{proof}

Consider the action Rota-Baxter Lie algebroid given in Proposition \ref{pro:actionRB}. By Theorem \ref{thm:rbpost}, there is a splitting post-Lie algebroid. It turns out that this post-Lie algebroid is the action post-Lie algebroid.

\begin{cor}\label{cor:actionPL}
Let $\phi:(\g,[\cdot,\cdot]_B)\lon \frkX(M)$ be an action of a Rota-Baxter Lie algebra  $(\g,[\cdot,\cdot]_\g,B)$ on a manifold $M$.
Let $((M\times \g)\oplus TM,[\cdot,\cdot]_\huaA,a_\huaA,\huaB)$ be the action Rota-Baxter Lie algebroid given in Proposition \ref{pro:actionRB}.
Then the splitting post-Lie algebroid of $((M\times \g)\oplus TM,[\cdot,\cdot]_\huaA,a_\huaA,\huaB)$ is exactly the action post-Lie algebroid of the splitting post-Lie algebra $(\g,[\cdot,\cdot]_\g,\rhd)$ of the Rota-Baxter Lie algebra  $(\g,[\cdot,\cdot]_\g,B)$.
\end{cor}
\begin{proof}
By Lemma \ref{lem:repRBPL}, $\phi:(\g,[\cdot,\cdot]_B)\lon \frkX(M)$ is also an action of the underlying post-Lie algebra $(\g,[\cdot,\cdot]_\g,\rhd)$ on $M$.
By Theorem \ref{thm:rbpost}, the induced post-Lie algebroid is given by $(M\times \g,[\cdot,\cdot]_\g,\rhd_A, a_A)$, where for all$~f,g\in C^\infty(M),u,v\in \g$,
\begin{eqnarray*}
  (fu)\rhd_A(gv)&:=&[fB(u)+f\phi(u),gv]_\huaA=fg[B(u), v]_\g+f\phi(u)(g)v=fg u\rhd v +f\phi(u)(g)v,\\
  a_A(u,m)&:=&a_\huaA\circ\huaB(u,m)=\phi(u)(m),
\end{eqnarray*}
which is exactly the action post-Lie algebroid of the post-Lie algebra  $(\g,[\cdot,\cdot]_\g,\rhd)$ underlying the Rota-Baxter Lie algebra  $(\g,[\cdot,\cdot]_\g,B)$.
 \end{proof}

\begin{rmk} \label{rk:actionPL2}
  A post-Lie algebra $(\g,[\cdot,\cdot]_\g,\rhd)$ gives rise to a subjacent Lie algebra $(\g,\Courant{\cdot,\cdot})$ in Eq~\eqref{eq:courant}. Similarly, it was shown in \cite{Munthe-Kaas-Lundervold} that a post-Lie algebroid $(A,[\cdot,\cdot]_A,\rhd_A, a_A)$ gives rise to a {\bf subjacent} Lie algebroid $(A,\Courant{\cdot,\cdot}_A, a_A)$, where the Lie bracket $\Courant{\cdot,\cdot}_A$ is given by
  $$
  \Courant{u,v}_A=[u,v]_A+u\rhd_A v-v\rhd_A u.
  $$
 It is straightforward  to see that when the post-Lie algebroid comes from the action $\phi$ of a post-Lie algebra $(\g,[\cdot,\cdot]_\g,\rhd)$, this subjacent Lie algebroid $(A,\Courant{\cdot,\cdot}_A, a_A)$ is exactly the action Lie algebroid   of the subjacent Lie algebra $(\g,\Courant{\cdot,\cdot})$.
See \cite[Theorem 4.4]{MSV} for the   necessary
and sufficient conditions for a Lie algebroid admitting a post-Lie algebroid structure.
\end{rmk}

By Corollary~\ref{cor:actionPL} and Remark~\ref{rk:actionPL2}, we have the following commutative diagram:
\begin{equation}
\begin{split}
\xymatrix{
 \text{Rota-Baxter Lie algebra} \ar^{\text{splitting}}[rr]   \ar_{\text{action}}[d] && \text{ post-Lie algebra}  \ar^{\text{action}}[d]\ar^{\text{subjacent}}[rr]&&\text{Lie algebra} \ar_{\text{action}}[d]\\
\text{Rota-Baxter Lie algebroid} \ar^{\text{splitting}}[rr] && \text{ post-Lie algebroid}\ar^{\text{subjacent}}[rr]&&\text{Lie algebroid.}
}
\end{split}
\mlabel{eq:actionsplitting}
\end{equation}

\subsection{Rota-Baxter Lie groupoids and integration of Rota-Baxter Lie algebroids}

We now introduce the notion of a Rota-Baxter Lie groupoid and show that the differentiation of a Rota-Baxter Lie groupoid is a Rota-Baxter Lie algebroid. Moreover, we introduce the notion of actions of Rota-Baxter Lie groups, from which we construct action Rota-Baxter Lie groupoids.

Recall that a groupoid~\cite{Mkz:GTGA} is a small category such that every arrow is invertible. Explicitly,
\begin{defi}
A {\bf groupoid} is a pair $(\G,M)$, where $M$ is the set of objects and $\G$ is the set of arrows, with the  structure maps
\begin{itemize}
\item two surjective maps $s,t: \G\longrightarrow M$, called the source map and target map, respectively;
\item  the multiplication $\cdot:\G^{(2)}\longrightarrow \G$, where $\G^{(2)}=\{(g_1,g_2)\in \G\times \G| s(g_1)=t(g_2)\}$;
\item  the inverse map $(\cdot)^{-1}:\G\longrightarrow \G$;
\item the inclusion map $\iota: M\longrightarrow \G$, called the identity map;
\end{itemize}
satisfying the following properties:
\begin{enumerate}
\item \rm{(associativity)} $(g_1\cdot g_2)\cdot g_3=g_1\cdot (g_2\cdot g_3)$, whenever the multiplications are well-defined;
\item \rm{(unitality)} $\iota(t(g))\cdot g=g=g\cdot \iota(s(g))$;
\item  \rm{(invertibility)} $g\cdot g^{-1}=\iota(t(g))$, $g^{-1}\cdot g=\iota(s(g))$.
\end{enumerate}
We also denote a groupoid by $(\G\rightrightarrows M,s,t)$ or simply by $\G$.

A {\bf Lie groupoid} is a groupoid such that both the set of objects and the set of arrows  are smooth manifolds, all structure maps are smooth, and the source and target maps are surjective submersions.
\end{defi}

 The tangent space of a Lie group at the identity has a Lie algebra structure. As its geometrization, on the vector bundle $\huaA:=\ker(t_*)|_M\longrightarrow M$ from a Lie groupoid,    there is a Lie algebroid structure defined as follows (\cite{Mkz:GTGA}):
the anchor map $a_\huaA:\huaA\longrightarrow TM$ is simply $s_*$ and the Lie bracket $[u,v]_\huaA$ is determined by
\[\overleftarrow{[u,v]_\huaA}=-[\overleftarrow{u},\overleftarrow{v}]_{T\G},\qquad \forall~u,v\in \Gamma(\huaA),\]
where $\overleftarrow{u}$ denotes the left-invariant vector field on $\G$ given by $\overleftarrow{u}_g=L_{g_*}u_{s(g)}$.

Denote by $\mathcal{O}_m:=s\circ t^{-1}(m)$ for $m\in M$. When $\mathcal{O}_m=M$ for $m\in M$, we call the Lie groupoid a {\bf transitive Lie groupoid}. The Lie algebroid associated to a transitive Lie groupoid is a transitive Lie algebroid.

Let $\G\rightrightarrows M$ be a transitive Lie groupoid. Its {\bf isotropy group} at $m\in M$ is defined to be \[\mathcal{H}_m:=s^{-1}(m)\cap t^{-1}(m).\] Denote by
$\mathcal{H}$ the bundle of Lie groups  over $M$ whose fiber at $m$ is the Lie group $\mathcal{H}_m$.

\begin{defi}
A {\bf Rota-Baxter operator} on a transitive Lie groupoid $\G\rightrightarrows M$ is a map $\B:\huaH\to \G$ covering the identity map on $M$ with respect to the target map, that is, $t\circ \B=t$,  satisfying
\begin{equation}\label{eq:conRBGoid}
  \B(g)\B(h)=\B(g\Ad_{\B(g)} h),\qquad \forall~g,h\in \huaH, ~\mbox{such  that}~ s(\B(g))=t(h).
\end{equation}
\end{defi}
The requirements $t(\B(g))=t(g), t(\B(h))=t(h)$ and $s(\B(g))=t(h)$ are to ensure that the multiplications appeared in the above formula are well defined.

\begin{pro}
 A Rota-Baxter operator $\B$ on a Lie groupoid $\G$ associates with $\huaH$ a new Lie groupoid structure,  called the {\bf \desc Lie groupoid}, given by
\begin{eqnarray*}
  \widetilde{s}(g)&:=&s(\B(g)),\\
   \widetilde{t}(g)&:=&t(g),\\
   g\star h&:=&g\Ad_{\B(g)} h, \quad \forall g,h\in \huaH ~s.t.~ s(\B(g))=t(h).
\end{eqnarray*}
\end{pro}
\begin{proof}
   It follows from a direct verification.
\end{proof}

Denote this Lie groupoid by $(\huaH,\star)$.

\begin{thm} \mlabel{thm:algebroid}
Let $(\G,\B)$ be a transitive Rota-Baxter Lie groupoid. Let $\huaA=\ker (t_*)|_M$ be the Lie algebroid of $\G$ and $\huaB:=\B_{*}:\ker(a_\huaA)\longrightarrow \huaA$ the tangent map of $\B$ at the identity. Then $(\huaA,\huaB)$ is a  Rota-Baxter Lie algebroid.
\end{thm}
\begin{proof}
We first claim that the Lie algebroid of $(\huaH,\star)$ is $(\ker(a_\huaA),[\cdot,\cdot]_\huaB,a_\huaA\circ \huaB)$, where
\[[u,v]_\huaB=[\huaB(u),v]_\huaA+[u,\huaB(v)]_\huaA+[u,v]_\huaA,\qquad \forall u,v\in \Gamma(\ker(a_\huaA)).\]
Note that $\ker (\tilde{t}_*)|_M=\ker (t_*)|_M=\ker (a_\huaA)$ and $\tilde{s}_*=s_*\circ \huaB$. Then following the same proof as in Proposition \ref{G*}, we find that the Lie bracket on $\ker (a_\huaA)$ of the Lie groupoid structure $\star$ is $[\cdot,\cdot]_\huaB$.
Moreover, it is direct to show that $\B:(\huaH,\star)\to \G$ is a Lie groupoid homomorphism, which induces a Lie algebroid homomorphism $\huaB:(\ker(a_\huaA), [\cdot,\cdot]_\huaB,a_\huaA\circ \huaB)\to \huaA$. This implies that
\[[\huaB(u),\huaB(v)]_\huaA=\huaB([\huaB(u),v]_\huaA+[u,\huaB(v)]_\huaA+[u,v]_\huaA).\]
So $(\huaA,\huaB)$ is a Rota-Baxter Lie algebroid.
\end{proof}

Assume that a Lie group $G$ left acts on a manifold $M$. Then we obtain a Lie groupoid $G\times M\rightrightarrows M$, whose source, target maps and multiplication are:
\[s(g,m):=m,\qquad t(g,m):=g\cdot m,\qquad (h,n)(g,m):=(hg,m),\]
for $n=g\cdot m$. This Lie groupoid is called the {\bf action Lie groupoid}, whose Lie algebroid is the action Lie algebroid.
Make the above phrase more precise, so that we have a commutative diagram
\begin{equation}
\begin{split}
\xymatrix{
	\text{actions of Lie groups on } M \ar^{\text{\qquad action}}[rr]   \ar_{\text{differentiation}}[d] && \text{Lie groupoids}  \ar^{\text{differentiation}}[d]\\
	\text{ actions of Lie algebras on } M \ar^{\text{\qquad action}}[rr] && \text{Lie algebroids.}
}
\end{split}
\mlabel{eq:lieaction}
\end{equation}

Suppose that the action is transitive. The isotropy group at $m\in M$ is $$\huaH_m=\{g\in G\,|\,g\cdot m=m\}.$$

Consider the Rota-Baxter operator on the action Lie groupoid.
\begin{pro}
A map   $\B: \huaH\to G\times M$ is a Rota-Baxter operator of weight $1$ on the action Lie groupoid  if and only if $\B_m: \huaH_m\to G\times \{m\}\cong G$ given by $\B(g,m)=(\B_m(g),\B_m(g)^{-1}m)$ satisfies
\begin{eqnarray}\label{actionRB}
\B_m(g)\B_n(h)=\B_m(g\Ad_{\B_m(g)} h),\qquad \forall g\in \huaH_m, h\in \huaH_n, n=\B_m(g)^{-1}m.
\end{eqnarray}
\end{pro}
Taking the differential, we get Proposition \ref{actioncase}.
\begin{proof}
   In fact, let $\tilde{g}=(g,m)\in \huaH_m$ and $\tilde{h}=(h,n)\in \huaH_n$ such that
\[s(\B(\tilde{g}))=\B_m(g)^{-1}m=t(\tilde{h})=n.\]
Note that $gm=m$ and $hn=n$. Then we have
\[\B(\tilde{g})\B(\tilde{h})=(\B_m(g),\B_m(g)^{-1}m)(\B_n(h),\B_n(h)^{-1}n)=(\B_m(g)\B_n(h),\B_n(h)^{-1}n)\]
and
\begin{eqnarray*}
\B(\tilde{g}\Ad_{\B(\tilde{g})} \tilde{h})&=&\B((g,m)(\B_m(g),\B_m(g)^{-1}m)(h,n)(\B_m(g)^{-1},m))\\ &=&
\B((g\B_m(g)h\B_m(g)^{-1},m))\\ &=&
(\B_m(g\Ad_{\B_m(g)} h),\B_m(g\Ad_{\B_m(g)} h)^{-1}m).
\end{eqnarray*}
Thus, $\B$ is a Rota-Baxter operator if and only if \eqref{actionRB} holds.
\end{proof}

Finally we study the groupoid analog of the action Rota-Baxter Lie algebroid given in Proposition \ref{pro:actionRB}. We give the definition of actions of Rota-Baxter Lie groups as follows.
\begin{defi}
  An {\bf action of a Rota-Baxter Lie group} $(G,\B_0)$ on a manifold $M$ is defined to be a right action $\B_1$ of the \desc  Lie group $(G,*)$  on $M$, where the group multiplication $*$ is given by \eqref{*}.
\end{defi}

Let $\B_1$ be an {action of a Rota-Baxter Lie group} $(G,\B_0)$ on a manifold $M$. Consider the  Lie groupoid $\G:=M\times G\times M\rightrightarrows M$,  where the source, target and the multiplication of this Lie groupoid are given by
\[s(m,g,n):=n,\qquad t(m,g,n):=m,\qquad (m',h,n')\cdot (m,g,n):=(m',hg,n),\]
when $n'=m$. It is easy to see that the bundle of isotropy groups is $ G\times M\to M$. The inclusion of $G\times M$ in $\G$ is given by $(g,m)\hookrightarrow (m,g,m)$.

Define a map $\B:  G\times M \to M\times G\times M$   by
\[\B(g,m)=(m,\B_0(g),\B_1(g,m)).\]

\begin{pro}
  Let $\B_1$ be an {action of a Rota-Baxter Lie group} $(G,\B_0)$ on a manifold $M$. Then the map $\B$ defined above is a Rota-Baxter operator on the Lie groupoid $\G=M\times G\times M\rightrightarrows M$.
\end{pro}

This Rota-Baxter Lie groupoid is called the {\bf action Rota-Baxter Lie groupoid}.

\begin{proof}
Let $\tilde{g}=(g,m), \tilde{h}=(h,n)\in G\times M$ such that
\[s(\B(\tilde{g}))=\B_1(g,m)=t(\tilde{h})=n.\] Then we have
\[\B(\tilde{g})\B(\tilde{h})=(m,\B_0(g),\B_1(g,m))(n,\B_0(h),\B_1(h,n))=(m,\B_0(g)\B_0(h),\B_1(h,\B_1(g,m)))\]
and
\begin{eqnarray*}
\B(\tilde{g}\Ad_{\B(\tilde{g})} \tilde{h})&=&\B((m,g,m)(m,\B_0(g),\B_1(g,m))(n,h,n)(\B_1(g,m),\B_0(g)^{-1},m))\\ &=&
\B((m,g\B_0(g)h\B_0(g)^{-1},m))\\ &=&(m,\B_0(g\Ad_{\B_0(g)} h),\B_1(g\Ad_{\B_0(g)} h,m)).
\end{eqnarray*}
Thus, $\B(\tilde{g})\B(\tilde{h})=\B(\tilde{g}\Ad_{\B(\tilde{g})} \tilde{h})$ holds if and only if
\begin{eqnarray*}
\B_0(g)\B_0(h)&=&\B_0(g\Ad_{\B_0(g)} h),\\
\B_1(h,\B_1(g,m))&=&\B_1(g\Ad_{\B_0(g)} h,m),
\end{eqnarray*}
which implies that $\B_0$ is a Rota-Baxter operator on the Lie group $G$ and $\B_1$ defines a right action of the Lie group $(G,*)$  on $M$.
\end{proof}

In summary, the diagram in \eqref{eq:lieaction} can be enriched to the diagram
\begin{eqnarray*}
\begin{split}
\xymatrix{
	\text{actions of Rota-Baxter Lie groups on } M \ar^{\text{\qquad action}}[rr]   \ar_{\text{differentiation}}[d] && \text{Rota-Baxter Lie groupoids}  \ar^{\text{differentiation}}[d]\\
	\text{actions of Rota-Baxter Lie algebras on } M \ar^{\text{\qquad action}}[rr] && \text{Rota-Baxter Lie algebroids.}
}
\end{split}
\mlabel{eq:rblieaction}
\end{eqnarray*}	

\noindent
{\bf Acknowledgements. } This research is supported by NSFC (11922110, 11771190, 11901568).

\end{document}